\documentclass[10pt,reqno]{amsart}

\usepackage[a4paper,left=35mm,right=35mm,top=30mm,bottom=30mm,marginpar=25mm]{geometry}

\usepackage{mathtools}
\mathtoolsset{showonlyrefs}
\usepackage{color}

\usepackage{amsmath}
\usepackage{amssymb}
\usepackage{amsthm}
\usepackage[latin1]{inputenc}
\usepackage{eurosym}

\usepackage{epsfig}
\usepackage{hyperref}
\usepackage{dsfont}

\usepackage[displaymath,mathlines]{lineno}

\allowdisplaybreaks

\usepackage{hyperref}

\usepackage{ifthen}
\usepackage[invisible]{MLatex}

\usepackage{graphicx}
\usepackage{caption}
\usepackage{subcaption}
\usepackage[nocompress, space]{cite}

\newcommand\sizefigure{0.40}

\makeindex

\renewcommand{\div}{\operatorname{div}}

\newcommand{\Rr}{{\mathbb{R}}}

\newcommand{\Nn}{{\mathbb{N}}}

\newcommand{\Tt}{{\mathbb{T}}}

\newcommand{\Hh}{{\overline{H}}}

\newcommand{\bx}{{\bf x}}

\newcommand{\bv}{{\bf v}}

\newcommand{\epsi}{\epsilon}
\def\d{{\rm d}}
\def\dx{{\rm d}x}

\def\dt{{\rm d}t}
\def\leq{\leqslant}
\def\geq{\geqslant}

\numberwithin{equation}{section}

\newtheoremstyle{thmlemcorr}{10pt}{10pt}{\itshape}{}{\bfseries}{.}{10pt}{{\thmname{#1}\thmnumber{
#2}\thmnote{ (#3)}}}
\newtheoremstyle{thmlemcorr*}{10pt}{10pt}{\itshape}{}{\bfseries}{.}\newline{{\thmname{#1}\thmnumber{
#2}\thmnote{ (#3)}}}
\newtheoremstyle{defi}{10pt}{10pt}{\itshape}{}{\bfseries}{.}{10pt}{{\thmname{#1}\thmnumber{
#2}\thmnote{ (#3)}}}
\newtheoremstyle{remexample}{10pt}{10pt}{}{}{\bfseries}{.}{10pt}{{\thmname{#1}\thmnumber{
#2}\thmnote{ (#3)}}}
\newtheoremstyle{ass}{10pt}{10pt}{}{}{\bfseries}{.}{10pt}{{\thmname{#1}\thmnumber{
A#2}\thmnote{ (#3)}}}

\theoremstyle{thmlemcorr}
\newtheorem{theorem}{Theorem}
\numberwithin{theorem}{section}

\newtheorem{proposition}[theorem]{Proposition}

\theoremstyle{thmlemcorr*}
\newtheorem{theorem*}{Theorem}
\newtheorem{lemma*}[theorem]{Lemma}
\newtheorem{corollary*}[theorem]{Corollary}
\newtheorem{proposition*}[theorem]{Proposition}
\newtheorem{problem*}[theorem]{Problem}
\newtheorem{conjecture*}[theorem]{Conjecture}

\theoremstyle{defi}
\newtheorem{definition}[theorem]{Definition}
\newtheorem{hyp}{Assumption}
\newtheorem{claim}{Claim}
\newtheorem{problem}{Problem}

\theoremstyle{remexample}
\newtheorem{remark}[theorem]{Remark}

\newtheorem{teo}[theorem]{Theorem}
\newtheorem{lem}[theorem]{Lemma}
\newtheorem{pro}[theorem]{Proposition}
\newtheorem{cor}[theorem]{Corollary}

\theoremstyle{ass}

\begin{document}

\title[Selection for MFGs]{ The selection problem for some first-order stationary Mean-field games}

\author{Diogo A. Gomes}
\address[D. A. Gomes]{
        King Abdullah University of Science and Technology (KAUST), CEMSE Division, Thuwal 23955-6900. Saudi Arabia, and  
        KAUST SRI, Center for Uncertainty Quantification in Computational Science and Engineering.}
\email{diogo.gomes@kaust.edu.sa}

\author{Hiroyoshi Mitake}
\address[H. Mitake]{
        Graduate School of Mathematical Sciences, University of Tokyo, 3-8-1 Komaba, Meguro-ku, Tokyo, 153-8914, Japan.}
\email{mitake@ms.u-tokyo.ac.jp}

\author{Kengo Terai}
\address[K. Terai]{
        Graduate School of Mathematical Sciences, University of Tokyo, 3-8-1 Komaba, Meguro-ku, Tokyo, 153-8914, Japan.}
\email{terai@ms.u-tokyo.ac.jp}


\thanks{
        D. Gomes was partially supported by King Abdullah University of Science and Technology (KAUST) baseline funds and KAUST OSR-CRG2017-3452.
        H. Mitake was partially supported by the JSPS grants: KAKENHI \#19K03580, \#19H00639, \#16H03948, \#17KK0093.        
        K. Terai was supported by King Abdullah University of Science and Technology (KAUST) through the Visiting Student Research Program (VSRP)
}
\date{\today}

\begin{abstract}

Here, we study the existence and the convergence of solutions for the vanishing discount MFG problem with a quadratic Hamiltonian. We give conditions under which the 
discounted problem has a unique classical solution and prove convergence 
of the vanishing-discount limit to a unique solution up to constants.
Then, we establish refined asymptotics for the limit. 
When those conditions do not hold, the limit problem may not have a unique solution
and its solutions may not be smooth,
as we illustrate in an elementary example. 
Finally, we investigate the stability of regular weak solutions and 
address the selection problem. Using ideas from Aubry-Mather theory, we establish a selection criterion for the limit.
\end{abstract}

\maketitle

\section{Introduction}



Mean-field games (MFG) model systems with many rational noncooperative players,
describe the player's optimal strategies and determine the statistical properties of their distribution.
These games are often determined by a system of a Hamilton-Jacobi equation coupled
with a transport or Fokker-Planck equation. 
In the study of stationary Hamilton-Jacobi equations, 
a standard method
to obtain a solution is to consider the vanishing discount problem. 
This was the strategy used originally in \cite{LPV} in the study of homogenization problems. 
For second-order MFG,
the existence of a solution 
for the discounted problem was shown, for example, in \cite{GMit} and \cite{EvGom} and for first-order
MFG in \cite{FG2}  in the sense of weak solutions and in \cite{GraCard} using variational methods.  
In the second-order case, the vanishing discount limit was studied 
in \cite{CardPorrLongTimeME}. 
In the first-order case, the theory is not as much developed and the vanishing discount limit has not 
been examined previously.   
Here, our goal is to 
study the limit behavior as $\epsilon\to 0$ of 
the following discounted first-order stationary
mean-field game.
\begin{problem}
\label{P1}
Let $\Tt^d$ be the $d$-dimensional flat torus identified with $[0,1]^d$. 
Let
$V:\Tt^d \to \Rr$, 
$V\in C^{1,\alpha}(\Tt^d)$,  $g:[0,\infty) \to \Rr\cup\{-\infty\}$, {$g\in C^{1,\alpha}((0,+\infty))$},
  with $g$ strictly increasing, and fix a discount rate, $\epsi>0$.  
 Find $u^\epsi, m^\epsi :\Tt^d \to \Rr$ with $m^\epsi (x)\geq
0$
such that 
\begin{equation}\label{DP}
        \begin{cases}
                &\epsi u^\epsi+\frac{1}{2}|Du^\epsi|^2+V(x)=g(m^\epsi) \quad\rm{in}\  \Tt^d, \\
                &\epsi m^\epsi -\mathrm{div}(m^\epsi Du^\epsi) = \epsi \quad \rm{in} \ \Tt^d. \\
        \end{cases}
\end{equation}
\end{problem}

We say that $(u^\epsilon, m^\epsilon)$ is a classical solution of the preceding problem
if {$u^\epsilon \in C^{2,\alpha}(\Tt^d)$ and $m^\epsilon \in C^{1,\alpha}(\Tt^d)$} with $m^\epsilon\geq 0$.  As 
 we show in Proposition \ref{density}, $m^\epsilon$ cannot vanish, hence, $m^\epsilon>0$. 
As in the case of Hamilton-Jacobi equations, 
we expect that, as 
$\epsilon\to 0$, the solutions of \eqref{DP}
converge, maybe through subsequences after adding a suitable constant to $u^\epsilon$, to a solution of 
 the following first-order MFG. 
 \begin{problem}
 \label{P2}
With $g$, $V$ as in Problem \ref{P1}, find
 	$u, m:\Tt^d \to \Rr$ with $m\geq
 	0$ and $\Hh\in \Rr$ such that 
\begin{equation}\label{CP}
        \begin{cases}
                &\frac{1}{2}|Du|^2+V(x)=g(m)+\bar{H} \quad\text{in}\  \Tt^d, \\
                & -\mathrm{div}(mDu) = 0 \quad \text{in} \ \Tt^d, \\
                & m(x)\geq 0,\ \ \int_{\Tt^d} m dx=1.
        \end{cases}
\end{equation}
\end{problem} 


Because \eqref{CP} is invariant under addition of constants to $u$, we can prescribe the additional normalization condition
\begin{equation}
\label{nu}
\int_{\Tt^d}u d x=0.  
\end{equation}

According to \cite{FG2} (also see \cite{FGT1}), Problem \ref{P2} admits weak solutions
under suitable polynomial growth conditions of $g$, see Corollary 6.3 in \cite{FG2}. Here, in Section
\ref{selection}, under a 
different set of hypothesis
and using a limiting
argument, we establish the existence of solutions for Problem \ref{P2}. 
A natural question in the analysis of the limit $\epsi\to 0$ is the {\em selection problem}; that is,  whether the sequence 
$(u^\epsi, m^\epsi)$ converges (not just whether a subsequence converges) and if so, what is the limit
among all possible solutions of \eqref{CP}. 
This matter is our main focus here. 

For Hamilton-Jacobi equations,  
the discounted problem corresponds to the following control problem. 
 Let ${\bf x}(t) \in \Rr^d$ be the state of an agent at the time $t$. This agent can change its state by choosing a control ${\bf v} \in L^{\infty}([0,\infty),\Rr^d)$. Thus, its trajectory,
  $\bf{x}(t)$,  is determined by  $\dot{\bf{x}}\rm(t)=\bf{v}\rm(t)$, with initial condition $\bx(t)=x\in \Tt^d$.
  The agent selects the control to minimize the cost functional 
\begin{equation*}
J(x; \bv)= \int^\infty_0 e^{-\epsi t}L({\bf x}(t),\dot{{\bf x}}(t)) \,\dt,
\end{equation*}
for a given Lagrangian, $L:\Tt^d \times \Rr^d \to \Rr$.  
 The  {\em value function}, $u^\epsi$, is given by
\begin{equation*}
u^\epsi(x)=\inf_{\bf{v}} J(x; \bv), 
\end{equation*}
where the infimum is taken over  $\bv\in L^\infty([0, +\infty), \Rr^d)$. 

The 
 {\em  Hamiltonian},\ $H:\Tt^d \times \Rr^d \to \Rr$, corresponding to this control problem is the Legendre transform of $L$;  that is, \[H(x,p)=\sup_{v \in \Rr^d}-p\cdot v-L(x,v).\]  
Under standard coercivity and convexity assumptions on $L$,
 $u^\epsi$ is the unique viscosity solution of the discounted Hamilton-Jacobi equation, 
\begin{equation}\label{D}
\epsi u^\epsi +H(x,Du^\epsi)=0 \quad\rm{in}\  \Tt^d.
\end{equation}
For coercive Hamiltonians, the results in \cite{LPV}  give that 
$\epsi u^\epsi$ is uniformly bounded and that $u^\epsilon$ is equi-Lipschitz for $\epsi>0$. Thus, $u^\epsi-\min_{\Tt^d}u^\epsi$ uniformly converges to a function, $u$, along subsequences, as $\epsi \to 0$. Moreover, $\epsi u^\epsi$ converges to a constant $-\tilde{H}$. By stability of viscosity solutions, $(u,\tilde{H})$ solves
the ergodic Hamilton-Jacobi equation
\begin{equation}\label{E}
H(x,Du)=\tilde{H}  \quad\rm{in}\  \Tt^d, 
\end{equation}
where the unknowns are $u:\Tt^d \to \Rr$ and $\tilde{H} \in \Rr$.
 However, the solution of \eqref{E} may not be unique. Hence, the solution constructed above could depend on the particular subsequence used to extract the limit.
 The study of the selection problem was started in \cite{MR2458239}
 using the discounted Mather measures introduced in \cite{CDG}. The main convergence
 result was established in \cite{MR3556524}. 
{Subsequently, several authors investigated 
 and extended those ideas in \cite{MR3475295}, \cite{MR3670619},
  \cite{MR3682741}, and \cite{MR3581314} .}
 Recently, the case of non-convex Hamiltonians was addressed in \cite{2016arXiv160507532G}.

In MFGs, we consider a large population of agents where 
each agent seeks to optimize an objective function. Here, however, the running cost depends on statistical information about the players, encoded in a probability density,  $m:\Tt^d\times [0,\infty)\to \Rr$. In the model discussed here, the Lagrangian is 
 $\hat{L}(x,p)=\frac{1}{2}|p|^2-V(x)+g(m)$
 and each agent seeks to 
  minimize the functional
\begin{equation*}
\hat{J}(x)= \int^\infty_0 e^{-\epsi t}\left[ \frac{1}{2}|\dot{{\bf x}}(t)|^2-V({\bf x}(t))+g\big(m({\bf x}(t),t)\big)\right] \,\dt.
\end{equation*}
Now, we suppose that the value function, $u^\epsi:=\inf_{{\bf v}}\hat{J,}$ is smooth.
Then, $u^\epsi$ solves the first equation in \eqref{DP} and the optimal control is given by ${\bf v}(t)=-Du^\epsi({\bf x}(t))$.
Because the players are rational, they use this optimal control. Here, $\epsi$ represents the rate at which players quit the game, which occurs at independent and
memoryless times. Furthermore, new players join the game randomly at a rate $\epsi$, as can be seen by looking at the right-hand side of the second equation in \eqref{DP}. Then, in the stationary configuration, the density, $m$, is determined by the second equation in \eqref{DP}. Without an inflow of players, the only non-negative solution is trivial, $m=0$.








The theory for second-order stationary MFG is now well developed 
and in many cases the existence of smooth solutions can be established, 
see for example
\cite{GM},
\cite{GPatVrt},
\cite{PV15},
or
\cite{EvGom}.
For logarithmic nonlinearities, 
the existence of smooth solutions was proven in 
\cite{evans2003some}. However, this is a special case; 
as shown in Section \ref{LU}, 
for first-order MFG, the existence of smooth solutions may not hold (see  also a detailed discussion in \cite{Gomes2016b} and \cite{GNPr216}). 
Thus, 
in general, we need to consider weak solutions, see \cite{EFGNV2017} or \cite{FG2}
for an approach using monotone operators and \cite{GraCard} for a variational approach. 

One of the difficulties of first-order stationary MFG is the lack of regularizing terms in both the Hamilton-Jacobi equation and in the transport equation.  Nonetheless, the MFG system behaves somewhat like an elliptic equation. Here, we explore this effect and obtain conditions under which 
Problem \ref{P1} has classical solutions. These conditions are given in the following two assumptions.

\begin{hyp}\label{osc}
 $g$ and $V$ satisfy that $g^{-1}\big(g(1)-\mathrm{osc}_{x \in \Tt^d}V(x)\big)>0$.
\end{hyp}

\begin{hyp}\label{beta}
There exist constants $C_1>0$, $C_2>0$ and $\beta \in \Rr$ such that for all $z \geq 0$,
\[ g'(z)\geq C_1z^\beta, \]
\[g(z)\leq C_2+zg(z).\]
\end{hyp}
An example that satisfies the preceding assumptions is the following: 
\[g(m)=m^\alpha \:\:(\alpha>0), \:\: V(x)=c\sin(2\pi x)\,\,(0<c<1/2),\]
{where $d=1$ and $V$ is extended by periodicity to $\Rr$.} 
The preceding two assumptions are used to obtain lower bounds on the density and can be interpreted as follows. Because $g$ is increasing agents want to avoid crowded areas and prefer areas with low density. 
However, if the oscillation of the potential is large, the trade-off between a low-density area with high potential and a high-density area with low potential may not pay-off. Hence, the control of
the oscillation of $V$ given in Assumption \ref{osc} implies that no point is totally avoided by the agents. 

As we mentioned previously, 
the two preceding assumptions imply the existence of a classical solution for Problem \ref{P1} as stated in the following theorem. 
\begin{teo}\label{result1}
Suppose that Assumptions \ref{osc} and \ref{beta} hold. Then, for each $\epsi>0$,  Problem \ref{P1} has a unique classical solution $(u^\epsi,m^\epsi)$ with $m^\epsi>0$.
\end{teo}

The proof of this theorem is given in Section \ref{pthm} using a continuation method combined with the a priori estimates in Section \ref{esti} and the DeGiorgi-Nash-Moser argument outlined in Section \ref{esti2}. 
As a corollary of the preceding theorem, we obtain our first convergence result.  
\begin{cor}\label{cor1}
Suppose that Assumptions \ref{osc} and \ref{beta} hold. Then,  Problem \ref{P2} has a unique classical solution $(u,m,\bar H)$, with $m>0$  and $\int_{\Tt^d} u dx=0$. Furthermore, 
let  $(u^\epsi,m^\epsi)$ solve Problem \ref{P1}. Then
\[u^\epsilon -\int_{\Tt^d} u^\epsilon \mbox{ }dx \to u \:\:\: \text{in}\  C^{2, \alpha}(\Tt^d), \quad m^\epsilon \to m \:\:\: \text{in}\  C^{1, \alpha}(\Tt^d), \quad \epsilon u^\epsilon \to -\Hh  \:\:\: \text{uniformly. } \]


\end{cor}
The proof of this corollary is given at the end of Section \ref{pthm}.

For second-order MFGs, the vanishing discount problem for mean-field games was addressed in \cite{CardPorrLongTimeME}. 
Inspired by the approach there, we consider
the following formal asymptotic expansion 
\begin{equation}
\label{aeeexp}
u^\epsi -\bar{H}/\epsi\sim u+\lambda+\epsi v, \qquad m^\epsi \sim m+\epsi \theta
\end{equation}
for the solution of Problem \ref{P1}. Using this expansion 
in \eqref{DP}, assuming that $(u,m,\lambda)$ solves Problem \ref{P2},  and 
matching powers of $\epsi$, we obtain the following problem that determines the
terms $\lambda$, $v$,  and $\theta$ in \eqref{aeeexp}. {To simplify the presentation, we discuss the case of $C^\infty$- solutions.}
\begin{problem}\label{P4} { Let $g$ be as in Problem \ref{P1} with $g\in C^\infty$} and let $(u,m)$
	be $C^\infty$- solutions of 
	 Problem \ref{P2} with $m>0$ and $\int u=0$. Find $v, \theta:\Tt^d \to \Rr$ and $\lambda \in \Rr$ such that 
	\begin{equation}\label{LLP}
	\begin{cases}
	&\lambda +u+Du\cdot Dv=g'(m)\theta \quad \rm{in} \ \Tt^d, \\
	& -\mathrm{div}(m Dv)-\mathrm{div}(\theta Du) = 1-m \quad \rm{in} \ \Tt^d. 
	\end{cases}
	\end{equation}
\end{problem}

\begin{remark}
	The normalization condition $\int u dx=0$ is required for the uniqueness of the constant $\lambda$. Given a solution of \eqref{LLP}, by adding a constant $\kappa$ to $u$ and subtracting $\kappa$ to $\lambda$, we produce another solution. 
\end{remark}

The existence of a solution to the preceding problem is established in Proposition \ref{lambda} 
in Section \ref{RA}. In that section,  
we prove the following improved asymptotic rate of convergence. 
\begin{theorem}
	\label{TT}	
	Suppose Assumption \ref{beta} holds. 
	Let  $(u^\epsi,m^\epsi)$ and $(u,m,\bar H)$, with $m>0$  and $\int u=0$, be classical solutions of, respectively, Problems \ref{P1} and \ref{P2}. Let $(v, \theta, \lambda)$ be the corresponding classical solution to Problem \ref{P4}.  Then,
	\[\lim_{\epsi \to 0} \left\| u^\epsi-\frac{\bar{H}}{\epsi}-u-\lambda\right\|_{\infty}+\|m^\epsi -m\|_{\infty}=0. \]
\end{theorem}

\begin{remark}
	\label{lbrem}
	The preceding theorem remains valid if we replace Assumption \ref{beta}
	with the weaker condition that for any $z_0>0$ there exists $\gamma(z_0)>0$ such that 
	\[
	g'(z)>\gamma(z_0)
	\]
	for all $z>z_0$. 
\end{remark}

In the last section of the paper,
Section \ref{selection}, 
we investigate the asymptotic behavior of $(u^\epsi, m^\epsi)$ as $\epsi \to 0$. Here, we work with weak solutions in the sense of the definition below, and we consider the case where uniqueness of solution for Problem \ref{P2} may not hold.
In this case, we replace Assumption \ref{osc} and \ref{beta} 
the following assumption that still allows the existence of solutions to be established. 
\begin{hyp}
	\label{H3}
	There exist positive constants, $c_1, c_2\geq 0$,
	and a positive real number, $\alpha$, with
	\begin{itemize}
		\item[-] $\alpha>0$ if $d\leq 4$
		\item[-] $\alpha>\frac{d-4}{d}$
		if $4<d\leq 8$
		\item[-] $\alpha>\frac{d-4}{2}$ if $d\geq 8$ 
	\end{itemize}
	such that 
	\[
	c_1 m^{\alpha-1}\leq g'(m)\leq c_2 m^{\alpha-1}
	\]
	for all $m>0$. 
\end{hyp}

\begin{remark}
	From the preceding hypothesis, we obtain that there exist positive 
	constants $\tilde c_1, \tilde c_2$ and $C$ such that 
	\[
	\tilde c_1 m^\alpha -C \leq g(m)\leq \tilde c_2 m^\alpha +C.
	\]
\end{remark}

Of course, 
if Assumption \ref{osc} does not hold, we cannot ensure the existence of smooth solutions to Problem \ref{P1}. Nonetheless, the existence of weak solutions for Problem \ref{P1} was proven in \cite{FG2}. More precisely, we
consider the following result. 
\begin{teo}[from \cite{FG2}]
	\label{ws}
	Suppose that Assumption \ref{H3} holds. Then, Problem \ref{P1} has a weak solution
	$(m^\epsilon, u^\epsilon)$ as follows. There exists a constant $C$, independent of $\epsilon$
	such that
	\begin{enumerate}
		\item $m^\epsilon\geq 0$ and $\int_{\Tt^d} m^\epsilon	dx =1$,
		\item $\|(m^\epsilon)^{\frac{\alpha+1}{2}}\|_{W^{1,2}(\Tt^d)}\leq C$,
		\item $\|u^\epsi-\int_{\Tt^d} u^\epsi dx \|_{W^{1,2}(\Tt^d)}\leq C$,
		\item $\left|\epsilon \int_{\Tt^d} u^\epsi dx\right|\leq C$, 
		\item $\|(m^\epsilon)^{\frac{\alpha+1} 2}Du^\epsilon\|_{BV(\Tt^d)}\leq C$. 
	\end{enumerate}
	Moreover, 
	\begin{equation}
	\label{subsol}
	-\epsilon u^\epsilon -\frac{|Du^\epsilon |^2}{2}-V(x)+g(m^\epsilon)\geq 0
	\end{equation}
	in the sense of distributions, 
	with 
	\begin{equation}
	\label{intweaksol}
	(-\epsilon u^\epsilon -\frac{|Du^\epsilon |^2}{2}-V(x)+g(m^\epsilon))m^\epsilon =0, 
	\end{equation}
	almost everywhere. Furthermore, 
	\begin{equation}
	\label{epweaksol}
	\epsilon m^\epsilon -\div ( m^\epsilon Du^\epsilon)=\epsilon, 
	\end{equation}
	in the sense of distributions and almost everywhere. 
\end{teo}
Similar techniques applied to Problem \ref{P2} yield the 
existence of
a number $\Hh$ and functions $(m,u)$ satisfying estimates 1-3 and 5 in Theorem \ref{ws} such that 
\begin{equation}
\label{ssole}
\Hh-\frac{|Du |^2}{2}-V(x)+g(m)\geq 0
\end{equation}
in the sense of distributions, 
with 
\begin{equation}
\label{asdfghjk}
(\Hh -\frac{|Du |^2}{2}-V(x)+g(m))m =0, 
\end{equation}
almost everywhere. Furthermore, 
\begin{equation}
\label{lpfp}
-\div ( m Du)=0, 
\end{equation}
in the sense of distributions and almost everywhere.

 When classical solutions are not available, we need to work with regular weak solutions, as defined next. 
\begin{definition}
A pair $(m^\epsilon, u^\epsilon)$ is a {\em regular weak solution} of Problem \ref{P1} if it satisfies 
\eqref{subsol}, \eqref{intweaksol} and \eqref{epweaksol} in the preceding theorem and, in particular, the same estimates 1-5 with the same constants. Similarly, a triple $(u,m,\Hh)$ is a {\em regular weak solution} of Problem \ref{P2} if it satisfies 
\eqref{ssole}, \eqref{asdfghjk},  \eqref{lpfp} and the estimates 1-5 in the preceding theorem with the same constants.
\end{definition}

In Section \ref{selection}, Proposition \ref{stabilitypro}, 
 we consider a sequence of regular weak solution of Problem \ref{P1} and show that, by extracting a subsequence if necessary, it converges to a regular weak solution of Problem \ref{P2}.  In particular, this approach gives 
 the existence a regular weak solution for Problem \ref{P2}.

Our selection result for regular weak solutions, proven in Section \ref{selection}, is the following theorem. 
\begin{teo}
\label{result2}
Let  $(u^\epsi,m^\epsi)$ be a regular weak solution of Problem \ref{P1}.
Suppose that $\langle u^\epsi \rangle \rightharpoonup \bar{u}$ in $H^1(\Tt^d)$ and that 
$m^\epsi\rightharpoonup \bar{m}$ weakly in $L^{1}(\Tt^d)$.
Let $(u,m)$ be a regular weak solution of Problem \ref{P2}. 
Then, 
\begin{equation}\label{converge}
\int_{\Tt^d} (g(m^\epsi)-g(m))(m^\epsi-m) dx\to 0
\end{equation}
and $\bar m=m$. 
Moreover, we have
\begin{equation}\label{criterion}
\int_{\Tt^d} \langle \bar u \rangle m dx
\leq 
\int_{\Tt^d} \langle u \rangle m dx, 
\end{equation}
where
\[
 \langle f\rangle = f(x)-\int_{\Tt^d}f dx.
 \]
\end{teo}

The proof of the preceding theorem relies on ideas from Aubry-Mather theory introduced in \cite{MR2458239}. The paper ends with a short example that illustrates the preceding result.

\section{Lack of uniqueness}
\label{LU}

Here, we examine the uniqueness of solutions of \eqref{CP}. First, we use the 
uniqueness method by Lasry-Lions \cite{ll3} to show that the probability density, $m$, is unique. Thus, failure of uniqueness for \eqref{CP} requires
multiplicity of solutions, $u$,  of the Hamilton-Jacobi equation. 
Second, we revisit an example from \cite{Gomes2016b}, where uniqueness does not hold. 
This example serves to illustrate the selection principle derived in Section \ref{selection}. 
\subsection{Lasry-Lions method}
\label{llm}
The monotonicity argument introduced by Lasry-Lions (see, \cite{ll1}  or the lectures \cite{LCDF}), can be used to prove the uniqueness of solution for MFGs in the time-dependent case and gives the uniqueness of $m$ in the { stationary problem}. Here, we apply this technique to Problem \ref{P2}.
Let $(u_1,m_2,\bar{H}_1)$ and $(u_2,m_2,\bar{H}_2)$ be classical solutions of \eqref{CP}. Then,
\begin{equation*}
\begin{cases}
&\frac{1}{2}|Du_1|^2-\frac{1}{2}|Du_2|^2+\bar{H}_1-\bar{H}_2=g(m_1)-g(m_2) \\
& -\mathrm{div}(m_1Du_1)+\mathrm{div}(m_2Du_2) = 0. \\ 
\end{cases}
\end{equation*}
Now, we multiply the first equation by $(m_1-m_2)$ and the second equation by $(u_1-u_2)$. Next, subtracting the resulting identities and integrating by parts, we obtain
\begin{equation}\label{identity}
\int_{\Tt^d}(m_1-m_2)(g(m_1)-g(m_2))+\frac{1}{2}(m_1+m_2)|Du_1-Du_2|^2 \,\dx=0.
\end{equation}
Accordingly, $m_1=m_2=m$ on $\Tt^d$ { because $g$ is strictly increasing}. Moreover, $Du_1=Du_2$ on $m>0$. Hence, classical solutions $(u,m,\bar{H})$ of \eqref{CP} with $m>0$ are unique up to an additive term in $u$. Uniqueness may fail if $m$ vanishes, as we show in Section \ref{aee}.
A similar proof gives that \eqref{identity} holds for the solutions of \eqref{DP}. By Lemma \ref{density}, $m^\epsi$ is positive. Hence, classical solutions of Problem \ref{P1} are unique.



\subsection{An explicit example}
\label{aee}

Here, we compute two distinct solutions of \eqref{CP}. In the example below, the existence of a unique  smooth solutions fails and $m$ vanishes at an interval.  
We show that
 $u$ is a Lipschitz viscosity solution and $m$ is a probability density.
These solutions are regular weak solutions as defined in Section \ref{WSS}. 
 
Let $g(m)=m$, $d=1,$ and $V(x)=\pi \cos (4\pi x)$. Then, \eqref{CP} becomes
\begin{equation}\label{BBB}
\begin{cases}
&\frac{1}{2}|u_x|^2+\pi\cos(4\pi x)=m+\bar{H} \quad\rm{in}\  \Tt, \\
& -(mu_x)_x = 0 \quad \rm{in} \ \Tt, \\
& m(x)\geq 0,\ \ \int m =1.
\end{cases}
\end{equation}
 From the second equation in \eqref{BBB}, $mu_x$ is constant. In particular, due to periodicity $u$ achieves a maximum or a minimum 
in $\Tt$. At this maximum or minimum point $u_x=0$. Accordingly, $mu_x=0$.
Thus, $u$ is constant on the set $m>0$.
 From the first equation in \eqref{BBB}
 and taking into account 
 that $\int_{\Tt} m dx=1$, we have $\bar{H}=0$
and, thus, 
 \[
  m(x)=(\pi \cos(4 \pi x))^+. 
  \]
 The preceding expression vanishes in an interval, as can be seen in Figure \ref{F1}.  
Set
\begin{align*}
(\hat{u}(x))_x=&\sqrt{2(-\pi \cos(4\pi x))^+}\cdot \chi_{\{\frac 1 8<x<\frac 1 4 \vee \frac 5 8<x<\frac 3 4\}}\\&-\sqrt{2(-\pi \cos(4\pi x))^+}\cdot \chi_{\{\frac 1 4<x<\frac 3 8 \vee \frac 3 4<x<\frac 7 8\}},
\end{align*}

and
\begin{equation*}
(\tilde{u}(x))_x=\sqrt{2(-\pi \cos(4 \pi x))^+}\cdot \chi_{\{\frac 1 8<x<\frac 3 8\}}-\sqrt{2(-\pi \cos(4\pi x))^+}\cdot \chi_{\{\frac 5 8<x<\frac 7 8\}},
\end{equation*}
where $\chi$ is the characteristic function. We observe that $(\hat{u},m,0)$ and $(\tilde{u},m,0)$ solve \eqref{BBB}. These two solutions are viscosity solutions - $\tilde u_x$ is continuous,  and
$u_x$ only has downward jumps, see Figures \ref{F2} and \ref{F3}.

\begin{IMG}
{"mplot", 
PP = If[# > 0, #, 0] &;
Plot[ PP[(Pi Cos[ 4 Pi x])], {x, 0, 1}, AxesLabel -> {"x", "m"}]
}	
\end{IMG}

\begin{figure}[htb!]
	\begin{centering}
		\includegraphics[width=0.4\textwidth]{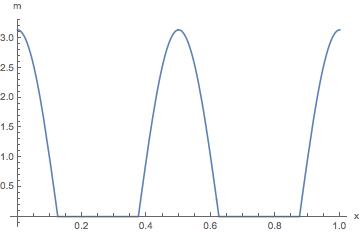}
		\caption{Density $m$ for \eqref{BBB} which exhibits areas with no agents.}
		\label{F1}
	\end{centering}
\end{figure}

\begin{IMG}
{"uxplot", 
Plot[Sqrt[2 PP[(-Pi Cos[ 4 Pi x])]] If[1/8 < x < 1/4 || 5/8 < x < 3/4, 1, 0]-Sqrt[2 PP[(-Pi Cos[ 4 Pi x])]] If[1/4 < x < 3/8 || 3/4 < x < 7/8, 1, 0], {x, 0, 1},
PlotPoints -> 200, MaxRecursion -> 10, AxesLabel -> {"x", Subscript["u", "x"]}]	
}
\end{IMG}
\begin{IMG}
	{"tildeuxplot", 
Plot[Sqrt[2 PP[(-Pi Cos[ 4 Pi x])]] If[1/8 < x < 3/8, 1, 0] - Sqrt[2 PP[(-Pi Cos[ 4 Pi x])]] If[5/8 < x < 7/8, 1, 0], {x, 0, 1}, PlotPoints -> 200, MaxRecursion -> 10, 
AxesLabel -> {"x", Subscript[OverTilde["u"], "x"]}]	}
\end{IMG}

\begin{figure}[h]
\centering	
	\begin{subfigure}[b]{\sizefigure\textwidth}
		\centering
		\includegraphics[width=\textwidth]{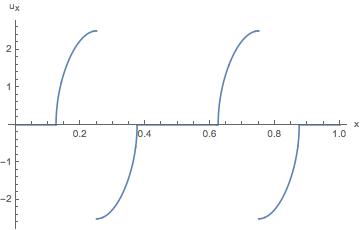}
		\caption{$u_x$}
		\label{F2}
	\end{subfigure}%
	\begin{subfigure}[b]{0.4\textwidth}
		\centering
		\includegraphics[width=\textwidth]{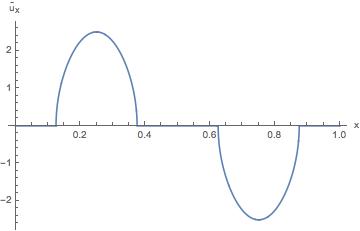}
		\caption{$\tilde u_x$}
		\label{F3}
	\end{subfigure}%
	\caption{Two distinct solutions, $u$ and $\tilde u$, of the Hamilton-Jacobi equation in \eqref{BBB}. Their gradients differ only when $m$ vanishes.}
\end{figure}

\section{Preliminary estimates}
\label{esti}

In this section, we establish preliminary a priori estimates for solutions of Problem \ref{P1}. To simplify the notation, we denote by $(u,m)$ 
a  
solution of Problem \ref{P1}, 
instead of $(u^\epsi,m^\epsi)$.
Here, we seek to establish bounds for $(u,m)$ that are uniform in $\epsi$. Accordingly,   
the bounds in this section depend only on the data, $g$, $V$, and $d$ but not on $\epsi$ nor on the particular solution.  
First, we show that  $m$ is a probability; that is, nonnegative and its integral is 1. Next, we establish a lower bound  and higher integrability for $m$. Finally, we prove Lipschitz bounds for $u$, which give the regularity of the solutions in the one-dimensional
case. The higher dimensional case requires further estimates that are addressed in the following section. 
\begin{pro}\label{density}
Let $(u,m)$ be a classical solution of Problem \ref{P1}. Then, for every $x\in \Tt^d$, $m(x)> 0$ and
\begin{equation}\label{AB}
 \int_{\Tt^d} m\,\dx =1.
\end{equation}
\end{pro}

\begin{proof}First, we show the positivity. Suppose that  $x_0\in \Tt^d$ is such that  $m(x_0)=\min_{x\in \Tt^d} m(x)=0$. At this point, the second equation in \eqref{DP} becomes
\[\epsi m(x_0)-Dm(x_0)Du(x_0)-m(x_0)\Delta u(x_0)=\epsi.\]
However, the left-hand side is $0$, which is a contradiction. To check \eqref{AB}, we integrate
the second equation in \eqref{DP} and use integration by parts. Then, we see that
\[ \epsi \int_{\Tt^d} m \dx = \epsi. \]
Thus, we get the conclusion.
\end{proof}
Next, we get a uniform lower bound for $m$.
\begin{pro}\label{lbd}
Suppose that Assumption \ref{osc} holds. Then, there exists a constant, $C>0,$ such that for any classical solution, $(u,m)$  of Problem \ref{P1}, we have
\begin{equation*}
\| \epsi u\|_{L^{\infty}(\Tt^d)}+\left\|\frac{1}{m}\right\|_{L^{\infty}(\Tt^d)} \leq C.
\end{equation*} 
\end{pro}
\begin{proof}
First, we bound 	
$\| \epsi u\|_{L^{\infty}(\Tt^d)}$. 	
Let $\tilde{x} \in \Tt^d$ be a minimum point of $u$. At this point, $Du(\tilde{x})=0$ and $\Delta u(\tilde{x}) \geq 0$. From the second equation in \eqref{DP}, we get
\begin{equation*}
m(\tilde{x})=\frac{\epsi}{\epsi-\Delta u(\tilde{x})}.
\end{equation*}
Since $m$ is positive, $\Delta u(\tilde{x})< \epsi$ and, thus, $m(\tilde{x})\geq 1$. Because $g$ is increasing, it follows from the first equation in \eqref{DP}
that \begin{equation}\label{ump}
\epsi u(\tilde{x})\geq g(1)-\max_{x \in \Tt^d}V(x).
\end{equation}
Next, let $\hat{x} \in \Tt^d$ be a maximum point of $u$. By an analogous argument, we get
\begin{equation*}
\epsi u(\hat{x}) \leq g(1)+\min_{x \in \Tt^d}V(x).
\end{equation*}
Thus, $\|\epsi u\|_{L^\infty(\Tt^d)} \leq C$.

Now, we address the lower bound for $m$.
By the first equation in  \eqref{DP} and \eqref{ump}, for all $x \in \Tt^d$, we have
\begin{align*}
g\big(m(x)\big)
&=\epsi u(x)+ \frac{1}{2}|Du(x)|^2+V(x)\\
&\geq \epsi u(\tilde{x})+\min_{x\in \Tt^d}V(x)
\geq g(1)-\mathrm{osc}V.
\end{align*}
Using Assumption \ref{osc}, we get the lower bound for $m$.
\end{proof}

In the following Lemma, we give an
upper bound for $m$.
\begin{lem}
Suppose that Assumptions \ref{osc} and \ref{beta} hold. Then, there exists a constant, $C>0,$ such that for any classical solution, $(u,m),$  of Problem \ref{P1}, we have
\begin{equation}\label{intm}
\int_{\Tt^d} m^{\beta+2} \,\dx \leq C.
\end{equation}
\end{lem}

\begin{proof}
First, we multiply the first equation in \eqref{DP} by $(1-m)$ and the second equation in \eqref{DP} by $u$. Integrating by parts and adding the resulting identities,  we have
\begin{equation}\label{iddd}
\int_{\Tt^d}\frac{1+m}{2}|Du|^2+mg(m) \,\dx=\int_{\Tt^d} (m-1)V+g(m) \,\dx.
\end{equation}
Using Assumption \ref{beta}, we get
\begin{equation*}
\int_{\Tt^d} mg(m)\,\dx \leq C.
\end{equation*}
On the other hand, in light of Proposition \ref{lbd}, there exists $0<m_0<m(x)$ for all $x \in \Tt^d$.
Furthermore, Assumption \ref{beta} guarantees that for all $t\geq m_0$,
\[g(t)\geq \frac{C}{\beta+1}t^{\beta+1}-\frac{C}{\beta+1}{m_0}^{\beta+1}+g(m_0).\]
Therefore, combining the preceding inequalities with \eqref{iddd}, we obtain \eqref{intm}.
\end{proof}

In the next proposition, we establish that $u$ is Lipschitz continuous and get uniform bounds for $m$ using 
a technique introduced in \cite{evans2003some}. 
\begin{pro}\label{ubd}
Suppose that Assumptions \ref{osc} and \ref{beta} hold. Let $d\geq 2$. Then, there exists a constant $C>0$ such that for any classical solution, $(u,m),$ of Problem \ref{P1}, we have
\begin{equation*}
\|Du\|_{L^{\infty}(\Tt^d)}+\|m\|_{L^{\infty}(\Tt^d)} \leq C.
\end{equation*}
\end{pro}

\begin{proof}
Take $p \geq \beta$. Multiplying the second equation in \eqref{DP} by $\mathrm{div}(m^pDu)$, we obtain
\begin{equation}\label{fp1}
\int_{\Tt^d} \epsi m\mathrm{div}(m^pDu) \mbox{ }dx=\int_{\Tt^d}\mathrm{div}(mDu)\mathrm{div}(m^pDu) \mbox{ }dx.
\end{equation}
Differentiating the first equation in \eqref{DP}, we get
\begin{equation}\label{HJ1} 
\sum_{j}u_{x_j}u_{x_j x_i}=g'(m)m_{x_i}-\epsi u_{x_i}-V_{x_i}.
\end{equation}
Next, we rewrite the right-hand side of \eqref{fp1} as follows:
\begin{align}\label{com}
&\int_{\Tt^d}\mathrm{div}(mDu)\mathrm{div}(m^pDu)
=\sum_{i,j} \int_{\Tt^d} (mu_{x_i})_{x_i}(m^pu_{x_j})_{x_j}  \mbox{ }dx=
\sum_{i,j} \int_{\Tt^d} (mu_{x_i})_{x_j}(m^pu_{x_j})_{x_i}   \mbox{ }dx \notag \\
&=\sum_{i,j} \int_{\Tt^d} pm^{p-1}(u_{x_i}m_{x_i})(u_{x_j}m_{x_j})
+pm^pm_{x_i}u_{x_j}u_{x_ix_j}+m^pu_{x_i}m_{x_j}u_{x_ix_j}+m^{p+1}u^2_{x_ix_j} \,\dx\notag \\
&=\sum_{i,j} \int_{\Tt^d} pm^{p-1}(u_{x_i}m_{x_i})(u_{x_j}m_{x_j})+(p+1)m^pm_{x_i}u_{x_j}u_{x_ix_j}+m^{p+1}u^2_{x_ix_j} \,\dx\notag \\
&=\int_{\Tt^d} pm^{p-1}|Dm\cdot Du|^2+(p+1)g'(m)m^p|Dm|^2 dx\\
&\qquad +\int_{\Tt^d} m^{p+1}\sum_{i,j} u^2_{x_ix_j}-(p+1)m^p\sum_j m_{x_j}(\epsi u_{x_j}+V_{x_j}) \,\dx, \notag 
\end{align}
using \eqref{HJ1} in the last line.
Combining \eqref{com} and \eqref{fp1}, we obtain
\begin{align*}
&\int_{\Tt^d} pm^{p-1}|Dm\cdot Du|^2+(p+1)g'(m)m^p|Dm|^2+m^{p+1}\sum_{i,j} u^2_{x_ix_j} \,\dx \\
&=\int_{\Tt^d} \epsi m\mathrm{div}(m^pDu)+(p+1)m^p\sum_j m_{x_j}(\epsi u_{x_j}+V_{x_j}) \,\dx \\
&=\int_{\Tt^d} \epsi pm^pDm\cdot Du+(p+1)m^pDm\cdot DV \,\dx \\
& \leq \int_{\Tt^d} \frac{\epsi p}{2}\big(m^{p+1}+m^{p-1}|Dm\cdot Du|^2\big)+(p+1)\left[\delta |Dm\cdot DV|^2m^{p+\beta}+C_{\delta}m^{p-\beta}\right]\dx, 
\end{align*}
where the last inequality follows from a weighted 
 Cauchy inequality with $\delta>0$ and $\beta$ is the exponent in Assumption \ref{beta}.
  For $\delta$ sufficiently small, there exists $C$ that does not depend on $p$, such that 
\begin{equation}\label{pestimate}
\int_{\Tt^d} m^{p+\beta}|Dm|^2 \,\dx \leq C\int_{\Tt^d}m^{p+1}\leq C\int_{\Tt^d}m^{p+\beta+2}\,\dx,
\end{equation}
where the last inequality is a consequence of $\int_{\Tt^d}m=1$ and $p+\beta+2>p+1$. 

For $d>2$, 
 $2^*=\frac{2d}{d-2}$ is the Sobolev conjugated exponent to $2$; if $d=2$, we use the convention that $2^*$ 
is an arbitrarily large real number. 
Using Sobolev's inequality and \eqref{pestimate}, we gather that
\begin{align*}
\left[\int_{\Tt^d} (m^{\frac{p+\beta+2}{2}})^{2^*}\dx \right]^{\frac{1}{2^*}}&\leq C\left[\int_{\Tt^d} m^{p+\beta+2}+|Dm^{\frac{p+\beta+2}{2}}|^2\,\dx \right]^{1/2}\\& \leq C(1+|p+\beta+2|)\left[ \int_{\Tt^d} m^{p+\beta+2}\,\dx \right]^{1/2}.
\end{align*}
Thus, there exists a positive constant $C>0$ such that for all $q\geq \beta+1$,
\begin{equation*}
\|m\|_{L^{\frac{2^{*}(q+1)}{2}}}\leq \left[C(1+q)\right]^{\frac{2}{q+1}}\|m\|_{L^{q+1}}.
\end{equation*} 

Next, we take $1<\theta<2^*/2$ and define $r_n=\theta^n+\beta+1$. By the previous Lemma that $\|m\|_{r_0}$ is bounded. Now we observe that $\frac{r_n}{r_{n+1}} <\frac{2}{2^*}$. Thus, for each $n \in \Nn$, there exists $0<\alpha_n<1$ satisfying
\[\frac{r_n}{r_{n+1}}=\alpha_n+\frac{1-\alpha_n}{2^*/2}.\]
By H\"older's inequality and the above estimate with $q+1=r_n$, we obtain
\begin{align*}
\|m\|_{r_{n+1}} 
&\leq \|m\|^{\alpha_n}_{r_n}\|m\|^{1-\alpha_n}_{2^{*}r_n/2} 
\leq \|m\|^{\alpha_n}_{r_n} \left\{ (C r_n)^{\frac{2}{r_n}}\|m\|_{r_n} \right\}^{1-\alpha_n} \\
&=(C r_n)^{\frac{2(1-\alpha_n)}{r_n}}\|m\|_{r_n}.
\end{align*}
Iterating the prior inequality, we get 
\begin{equation*}
\|m\|_{r_{n+1}} \leq \|m\|_{r_0} \prod^n_{i=0} (C r_i)^{\frac{2(1-\alpha_i)}{r_i}}.
\end{equation*}
The right-hand side is bounded uniformly in $n \in \Nn$ because
\[
\log \left(\prod^n_{i=0} (C r_i)^{\frac{2(1-\alpha_i)}{r_i}}\right)\leq \sum^n_{i=0}\frac{2}{r_i}\left[C+\log(r_i)\right]<+\infty.
\]
Hence, $\|m\|_\infty$ is bounded.
According to the first equation in \eqref{DP}, and using the bound for $\epsi u$ in Proposition \ref{lbd}, we obtain that $\|Du\|_\infty$ is also bounded.
\end{proof}

When $d=1$, we can improve the preceding results to show that $m$ is bounded, as shown in the next proposition. The case $d\geq 2$ is discussed in the next section. 

\begin{pro}
Suppose that Assumption \ref{osc} holds. Let $d=1$. Then, there exists a constant, $C>0,$ such that  for any classical solution, $(u,m)$, of \eqref{DP}, we have
\begin{equation*}
\|m\|_{L^\infty(\Tt)}+\|u_x\|_{L^\infty(\Tt)}\leq C.
\end{equation*}
\end{pro}
\begin{proof}
Multiplying the first equation by $m_{xx}$ and the second by $u_{xx}$, we obtain
\begin{align*}
\begin{cases}
&\epsi um_{xx}+\frac{1}{2}m_{xx}u_x^2+m_{xx}V=g(m)m_{xx}\\
&\epsi mu_{xx}-u_{xx}(mu_x)_x=\epsi { u_{xx}}.
\end{cases}
\end{align*}
Next, we subtract these equations and integrate by parts to get 
\begin{align*}
\int_{\Tt} mu_{xx}^2+g'(m)m_x^2\,\dx=\int_{\Tt} m_xV_x dx\leq \delta \int_{\Tt} m_x^2 dx+\frac{1}{4\delta}\int_{\Tt} V_x^2 dx,
\end{align*}
using a weighted Cauchy-Schwarz inequality with $\delta >0$.
Because $m$ is bounded by below, taking $\delta>0$ sufficiently small,  $\|m_x\|_{L^2(\Tt)}$ and $\|u_{xx}\|_{L^2(\Tt)}$ are bounded. Thus, we get the desired result.
\end{proof}

\begin{pro}\label{regularity1d}
Suppose that Assumption \ref{osc} holds. Let $d=1$. Then, there exists a constant $C>0$ such that for any classical solution, $(u,m),$ of Problem \ref{P1}, we have
\begin{equation}\label{CCCX}
\|u_x\|_{C^{1,\alpha}(\Tt)}+\|m\|_{C^{1,\alpha}(\Tt)} \leq C.
\end{equation}
\end{pro}
\begin{proof}
Differentiating the first equation in \eqref{DP}and multiplying by $m$, we get
\begin{equation}
\epsi mu_x+u_xmu_{xx}+mV_x=g'(m)mm_x.
\end{equation}
Solving the second equation in \eqref{DP} for $mu_{xx}$ and substituting in the above identity, we have
\begin{equation}\label{CCCY}
m_x=\frac{2\epsi mu_x-\epsi u_x+mV_x.}{(u_x^2+g'(m)m)}.
\end{equation}
Because $m$ is bounded by below, the denominator in the preceding expression does not vanish. Thus, from the previous Proposition, the right-hand side is bounded. Accordingly, $\|m_x\|_{L^{\infty}(\Tt)}$ is bounded. Returning to the second equation \eqref{DP}, we see that  $\|u_{xx}\|_{L^{\infty}(\Tt)}$ is bounded. {Returning to \eqref{CCCY}, we see that $\|m\|_{C^{1,\alpha}(\Tt)}$ is bounded. Thus, from the second equation in \eqref{DP}, we gather that $\|u_x\|_{C^{1,\alpha}(\Tt)}$ is bounded.}
\end{proof}

\section{Estimates in higher dimensions}
\label{esti2}

Now, we obtain additional estimates for the solutions of \eqref{DP} in the case $d\geq2$. 
As in the previous section, 	
to simplify the notation, we omit the $\epsilon$ in 
$(u^\epsi,m^\epsi)$
and denote by $(u, m)$ 
a  
solution of Problem \ref{P1}.
First, by solving the first equation in \eqref{DP} for $m$, we get
   \begin{equation*}
   m=g^{-1}\left(\epsi u +\frac{1}{2}|Du|^2+V \right).
   \end{equation*}
Next, replacing the resulting expression into the second equation in \eqref{DP}, we obtain
\begin{equation}\label{EL}
 \mathrm{div}\left[g^{-1}\left(\epsi u+\frac{|Du|^2}{2}+ V\right)Du\right]-\epsi \left[g^{-1}\left(\epsi u+\frac{|Du|^2}{2}+ V\right)-1\right]=0.
 \end{equation} 
Here, we apply the DeGiorgi-Nash-Moser regularity method to \eqref{EL} to obtain our estimates. 

We begin by selecting $k$ with $1\leq k \leq d$. Differentiating \eqref{EL} with respect to $x_k$, we conclude that $v=u_{x_k}$ solves
\begin{equation}\label{MMDP}
(a^{ij}v_{x_j})_{x_i}=\phi_{x_k}+\psi_{x_i},
\end{equation}
where
\begin{equation*}
a^{ij}(x)=g^{-1}\left(\epsi u+\frac{1}{2}|Du|^2+V\right)\delta_{ij}+(g^{-1})'\left(\epsi u+\frac{1}{2}|Du|^2+V\right)u_{x_i}u_{x_j},
\end{equation*}
\begin{equation*}
 \phi(x)=\epsi \left[{g^{-1}(\epsi u+\frac{1}{2}|Du|^2+V)-1}\right],
\end{equation*}
\begin{equation*}
\psi(x)=(g^{-1})'(\epsi u+\frac{1}{2}|Du|^2+V)(\epsi u_{x_i}u_{x_k}+u_{x_i}V_{x_k}), 
\end{equation*}
and $\delta_{ij}=1$ if $i=j$ and $\delta_{ij}=0$ otherwise. 
Because of Propositions \ref{lbd} and \ref{ubd}, 
there exists a constant, $C>0$, such that for any classical solution, $u$, of \eqref{EL}, we have $\|\epsi u\|_{\infty} +\|Du\|_{\infty} \leq C$. Hence, we get
\begin{equation}
\label{limfb}
\| \phi \|_{\infty}+\| \psi \|_{\infty} \leq C. 
\end{equation}
Moreover, using again Propositions \ref{lbd} and \ref{ubd}, we see that there exists a constant $\lambda>0$ such that  for all $\xi \in \Rr^d$, we have
\begin{equation}\label{elliptic}
\lambda|\xi|^2\leq \sum^d_{i,j=1} a^{ij}(x)\xi_i \xi_j \leq \frac{1}{\lambda} |\xi|^2.
\end{equation}

Next, we prove that $v$ is H\"older-continuous and, thus, get higher regularity for $u$.

\begin{pro}\label{regularity}
Suppose that Assumptions \ref{osc} and \ref{beta} hold. Let $d\geq2$. Then, there exist constants, $C>0$ and $0<\alpha<1$, such that for any classical solution, $u$, of \eqref{EL}, we have
\begin{equation*}
\|Du\|_{C^{1,\alpha}(\Tt^d)} \leq C.
\end{equation*}
\end{pro}
\begin{proof}
Take $R>0$. Let $v$ solve \eqref{MMDP}. Write $v=z+w$ where $z$ is a solution of 
\begin{equation}\label{eq1}
(a^{ij}z_{x_j})_{x_i}=\phi_{x_k}+\psi_{x_i},
\end{equation}
in $B_{2R}$ and $z=0$ on $\partial B_{2R}$. Therefore, $w$ solves
\begin{equation}\label{eq2}
(a^{ij}w_{x_j})_{x_i}=0,
\end{equation}
in $B_{2R}$ with $w|_{B_{2R}}=v$. 

We begin by establishing the following claim.
\begin{claim}
For $d>2$, 	
there exists a constant, $C>0$, that depends only on the bounds in 
\eqref{elliptic} such that 
\begin{equation*}
\| z \|_{L^\infty(B_{2R})} \leq CR \quad \text{for any } R>0.
\end{equation*}
\end{claim}
\begin{remark}
If $d=2$, we get $\| z \|_{L^\infty(B_{2R})} \leq \bar{C}R^\kappa$, for any $\kappa<\frac 1 2$.
This difference is due to the exponent $2^*$ in dimension $2$ being replaced by an 
arbitrarily large constant. 
The argument that follows
needs to be adapted accordingly, namely the bound in \eqref{Mb} below, but the key steps remain unchanged. This case will be omitted. 
\end{remark}

Let $k\geq0$. By multiplying \eqref{eq1} by $(z-k)^+$ and integrating by parts, we get
\begin{equation}
\label{opop}
\int_{B_{2R}} a^{ij}z_{x_j}(z-k)^+_{x_i}\,\d x=\int_{B_{2R}}\phi(z-k)^+_{x_k}+\psi(z-k)^+_{x_i}\,\d x.
\end{equation}
Set
\begin{equation*}
A(k)=\{ z>k\}\cap B_{2R}.
\end{equation*}
It suffices to prove that we can choose a constant $C>0$ satisfying $|A(CR)|=0$. 
Because $(z-k)^+_{x_i}=0$ on $A(k)^C$ and $(z-k)^+_{x_i}=z_{x_i}$ on $A(k)$, 
we obtain from \eqref{opop}
that
\begin{equation*}
\int_{A(k)} a^{ij}z_{x_j}z_{x_i}\,\dx=\int_{A(k)}\phi z_{x_k}+\psi z_{x_i}\,\dx.
\end{equation*}
In view of \eqref{elliptic} and the bounds in \eqref{limfb}, we get 
\begin{equation*}
 \int_{A(k)}|Dz|^2\,\dx \leq C |A(k)| .
\end{equation*}
Next, using Sobolev's inequality and taking into account that 
$(z-k)^+=0$ on $\partial B_{2R}$, we conclude that, for any $h>k$,
\begin{align*}
(h-k)^2|A(h)|^{2/2^*} &
\leq \left[ \int_{A(h)} [(z-k)^+]^{2^*}\,\dx\right]^{2/2^*}
\leq \left[\int_{B_{2R}} [(z-k)^+]^{2^*}\,\dx \right]^{2/2^*}\\
&\leq C\int_{B_{2R}} |D(z-k)^+|^2\,\dx \leq C\int_{A(k)}|Dz|^2\,\dx.
\end{align*}
Combining the two preceding estimates, we obtain
\begin{equation*}
|A(h)|\leq \frac{C|A(k)|^{\frac{2^*}{2}}}{(h-k)^{2^*}}.
\end{equation*}
Next, we take a sequence $k_n=M\left(1-\frac 1 {2^n}\right)$, where 
\begin{equation*}
M=\left(C|A(0)|^{\frac{2^*}{2}-1} 2^{\frac{(2^*)^2}{2^*-2}}\right)^{\frac{1}{2^*}}. 
\end{equation*}
Using the above estimate, we obtain 
\begin{equation*}
|A(k_{n+1})|\leq \frac{C}{(k_{n+1}-k_n)^{2^*}} |A(k_n)|^{{2^*}/2}\leq C\frac{2^{2^*(n+1)}}{M^{2^*}}|A(k_n)|^{{2^*}/2}.
\end{equation*}
We now prove by induction that
\begin{equation*}
|A(k_n)|\leq |A(0)|2^{-n\mu},
\end{equation*}
where $\mu=\frac{2^*}{\frac{2^*}{2}-1}$. The case $n=0$ is clear.
Assume our claim holds for some $n$, 
we have to check that it holds for $n+1$. 
We have
\begin{align*}
|A(k_{n+1})|
&\leq C\frac{2^{2^*(n+1)}}{M^{2^*}} \big(|A(0)|2^{-n\mu}\big)^{2^*/2}= |A(0)|2^{-(n+1)\mu},
\end{align*}
using our choice for $M$ and $\mu$. 

Finally, by considering the limit $n \to \infty$, we get $|A(M)|=0$. If $d>2$, we have
\begin{equation}
\label{Mb}
M=\hat{C}|A(0)|^{1/d}\leq \hat{C}|B_{2R}|^{1/d}=\bar{C}R.
\end{equation}
Thus, we get Claim 1.

Next, using the ellipticity bounds in \eqref{elliptic}, we apply the DeGiorgi-Nash-Moser estimate (see, \cite{GilTru}, Theorem $8.22$) to \eqref{eq2} to establish the following claim.
\begin{claim}
We have 
\begin{equation*}
\mathrm{osc}(\frac{R}{2},w)\leq \eta \, \mathrm{osc}(R,w),
\end{equation*}
for some constant $0<\eta<1$, where we denoted
\begin{equation*}
\mathrm{osc}(R,w):=\sup_{B_R}w-\inf_{B_R}w.
\end{equation*}
\end{claim}

Combining the two preceding claims, we obtain the following estimate:
\begin{align}\label{FF}
\mathrm{osc}(R/4,v)
&\leq CR+\mathrm{osc}(R/4,w) \leq \hat{C}R+\eta \, \mathrm{osc}(R,v).
\end{align}
\begin{claim}
\label{c3}
There exist constants $C>0$ and $0<\alpha<1$ such that for all $0<R<1$, we have
\begin{equation*}
\mathrm{osc}(R,v)\leq CR^{\alpha}.
\end{equation*}
\end{claim}
Set
\begin{equation*}
M_n=\sup_{\frac{1}{4^{n+1}}\leq R\leq \frac{1}{4^n}} \frac{\mathrm{osc}(R, v)}{R^{\alpha}},
\end{equation*}
where $\alpha$ satisfies $0<\alpha<-\frac{\log \eta}{\log 4}$.

Here, we prove by induction that there exist $\mu>1$ satisfying 
\begin{equation}\label{GG}
M_n \leq M\mu^{-n}
\end{equation}
for some sufficiently large $M>0$.
We choose $\mu$ satisfying 
\begin{equation}
\label{muchoice}
1<\mu<\min(4^{1-\alpha},\frac{1}{4^\alpha \eta}), 
\end{equation}
and then we choose $M> \frac{2}{4^\alpha}\|v\|_{L^\infty(\Tt^d)}$ and such that 
\begin{equation}
\label{Mchoice}
\left[ \hat{C}\left(\frac{\mu}{4^{1-\alpha}}\right)^{n+1}+\eta 4^\alpha M\mu\right] \leq M.
\end{equation}
The prior choice of $M$ is possible due to \eqref{muchoice}. 

For $n=0$, $M_0 \leq \frac{2}{4^\alpha}\|v\|_{L^\infty(\Tt^d)}\leq M$.
Next, we assume that \eqref{GG} holds for some $n\geq 0$ and verify that it also holds for $n$ replaced by 
$n+1$. Using \eqref{FF}, we have that
\begin{equation*}
M_{n+1}\leq \left[ \hat{C}\left(\frac{\mu}{4^{1-\alpha}}\right)^{n+1}+\eta 4^\alpha M\mu\right] \mu^{-(n+1)}.
\end{equation*}
Using the defining property of $M$, \eqref{Mchoice}, we get \eqref{GG}.

Finally, for $0<R<1$, combining \eqref{FF} and \eqref{GG}, we obtain
\begin{equation*}
\mathrm{osc}(R,v) \leq (4\hat{C}+\eta 4^\alpha M)R^\alpha, 
\end{equation*}
which establishes the claim.

\begin{claim}
\begin{equation*}
\|Du\|_{C^{1,\alpha}(\Tt^d)} \leq C.
\end{equation*}
\end{claim}
Due to Claim \ref{c3},  we get $\|Du\|_{C^{0,\alpha}(\Tt^d)} \leq C$. Since \eqref{MMDP} is a uniformly elliptic equation and the H\"older-norm of the coefficients is bounded. Therefore, it follows from  Schauder's estimate that $\|Dv\|_{C^{0,\alpha}(\Tt^d)} \leq C$. Hence, we conclude that  $\|Du\|_{C^{1,\alpha}(\Tt^d)}\leq C$.
\end{proof}

\section{Existence of solution for the discounted problem}\label{pthm}
Here, we prove the existence of classical solution for \eqref{EL}, 
which is 
equivalent to \eqref{DP}, using the continuation method. We begin by defining an operator, $J: C^{2,\alpha}(\Tt^d) \times [0,1] \to C^{0,\alpha}(\Tt^d)$, by
\begin{equation*}
J(u,\lambda)=
\mathrm{div}\left[g^{-1}\left(\epsi u+\frac{|Du|^2}{2}+\lambda V\right)Du\right]-\epsi \left[ g^{-1}\left(\epsi u+\frac{|Du|^2}{2}+\lambda V\right)-1\right].
\end{equation*}
We set
\begin{equation}
\label{lamb}
 \Lambda=\big\{\lambda \in [0,1]:
\exists u_\lambda \in C^{2,\alpha}(\Tt^d): \:J(u_\lambda ,\lambda)=0 \big\}.
\end{equation}
We claim that $\Lambda=[0,1]$. First, we observe that  $0\in \Lambda$. In fact, for $u_0 \equiv \epsi^{-1}g(1)$, we have $J(u_0,0)=0$.  Accordingly, $\Lambda$\ is non-empty. Thus, it suffices to check that $\Lambda$ is relatively open and closed in $[0,1]$, to get $\Lambda=[0,1]$. In the next proposition, 
we verify that $\Lambda$ is a closed set. 

\begin{pro}
Consider the setting of Problem \ref{P1} and suppose that Assumptions \ref{osc} and \ref{beta} hold.
Let $\Lambda$ as in \eqref{lamb}. Then, $\Lambda$ is relatively closed in $[0,1]$. 
\end{pro}

\begin{proof}
Fix a sequence $\lambda_n \in \Lambda$ converging to $\lambda\in [0,1]$
as $n \to \infty$. We must show that $\lambda \in \Lambda$. For that,
take $u_{\lambda_n}$ satisfying $J(u_{\lambda_n},\lambda_n)=0$. The a priori bounds in  Proposition \ref{regularity1d} (for $d=1$)
or Proposition \ref{regularity} (for $d>1$)
 guarantee that there exists a subsequence of $\{ u_{\lambda_n} \}_{n \in \Nn} $ converging to some $u \in C^{2,\alpha}(\Tt^d)$. By passing to the limit, we  conclude that $J(u,\lambda)=0$. Accordingly,  $\lambda \in \Lambda$.
\end{proof}

Now, we show that $\Lambda$ is relatively open. For each $\lambda \in \Lambda$, let $u_\lambda\in C^{2,\alpha}(\Tt^d)$ solve $J(u_\lambda,\lambda)=0$
and set
\begin{equation}
\label{mlam}
m_\lambda=g^{-1}\left(\epsi
u_\lambda+\frac{|Du_\lambda|^2}{2}+\lambda V\right).
\end{equation}
We consider the linearization
of $J$\ around this solution  and define $L_\lambda: C^{2,\alpha}(\Tt^d)
 \to  C^{0,\alpha}(\Tt^d)$ for $\phi \in C^{2,\alpha}(\Tt^d)$ by
\begin{align}
\label{ll}
&L_\lambda(\phi)=\frac{\partial J}{\partial \mu}(u_\lambda+\mu \phi,\lambda)\bigg|_{\mu=0}
\\\notag
&= \mathrm{div}\left[m_\lambda D\phi +(g^{-1})'(g(m_\lambda))(\epsi \phi+Du_\lambda
\cdot D\phi)Du_\lambda \right]-\epsi (g^{-1})'(g(m_\lambda))(\epsi \phi
+Du_\lambda \cdot D\phi). 
 \end{align}

\begin{lem}
	\label{isolem}
Consider the setting of Problem \ref{P1}.	
	Let $u_\lambda\in C^{2,\alpha}(\Tt^d)$ solve $J(u_\lambda,\lambda)=0$ and
	let $L_\lambda$ be given by \eqref{ll}. Then,
$L_\lambda$ is an isomorphism between $C^{2,\alpha}(\Tt^d)
 $ and $C^{0,\alpha}(\Tt^d)$ .
\end{lem}
\begin{proof} We must prove that for any $\xi \in C^{0,\alpha}(\Tt^d)$, 
the equation
\begin{align}\label{LP}
&\mathrm{div}\left[m_\lambda D\phi +(g^{-1})'(g(m_\lambda))(\epsi \phi+Du_\lambda
\cdot D\phi)Du_\lambda \right]\\\notag &-\epsi (g^{-1})'(g(m_\lambda))(\epsi \phi
+Du_\lambda \cdot D\phi)=\xi
\end{align}
has a unique solution, $\phi \in C^{2,\alpha}(\Tt^d)$. We define $B_\lambda:
H^1(\Tt^d) \times H^1(\Tt^d) \to \mathbb{R}$ by
\begin{align*}
 B_\lambda[v,w]
=&\int_{\Tt^d} Dw\left[ m_\lambda Dv +(g^{-1})'(g(m_\lambda))(\epsi v+Du_\lambda
\cdot Dv)Du_\lambda \right] \\&+\epsi\int_{\Tt^d} w(g^{-1})'(g(m_\lambda))(\epsi v +Du_\lambda
\cdot Dv)\dx.
\end{align*}
Note that if $v$ and $w$ are smooth, $B_\lambda$ becomes
$B_\lambda [v,w]=(L_\lambda(v),w)_{L^2}$.
Using H\"older inequality, we see that $B_\lambda$ is bounded. 
Now, using Riesz's Representation Theorem, we see that there exists a bounded
linear operator $A: H^1(\Tt^d) \to H^1(\Tt^d)$ such that $B_\lambda [v,w]=(Av,w)$
for all $w \in H^1(\Tt^d)$. We divide the rest of the proof in the following three claims. 

\begin{claim}
There exists a constant, $c>0,$ such that $ \|Av\|_{H^1(\Tt^d)} \geq c\|v\|_{H^1(\Tt^d)}$ for all $v\in
H^1(\Tt^d)$.
\end{claim}
We establish this claim by contradiction. 
For that, suppose that 
there exists $\{v_n\}_{n\in \Nn} \subset H^1(\Tt^d)$ with
$\|v_n\|_{H^1(\Tt^d)}=1$ and $Av_n \to 0$.
Then,
\[
B_\lambda[v_n,v_n]=(Av_n,v_n) \to 0.
\]
Next, we have
\begin{equation*}
B_\lambda [v_n,v_n]=\int_{\Tt^d}  m_\lambda |Dv_n|^2+(g^{-1})'(g(m_\lambda))\{
\epsi v_n+Du_\lambda \cdot Dv_n\}^2 \dx.
\end{equation*}
Since $m_\lambda$ and $(g^{-1})'$ are positive, we see that $Dv_n \to 0$ and
$\epsi v_n+Du_\lambda \cdot Dv_n \to 0$ in $L^2(\Tt^d)$. Hence, we can construct
a subsequence $\{v_{n_k}\}_{k\in \Nn}$  satisfying $v_{n_k} \to 0$ in
$H^1(\Tt^d)$, which contradicts
$\|v_n\|_{H^1(\Tt^d)}=1$.

\begin{claim}
The range of $A$, $ R(A)$, is closed and $R(A)=H^1(\Tt^d)$.
\end{claim}
Take a sequence $\{z_n\}_{n\in \Nn} \subset R(A)$ that converges to $z \in E$. To prove the first part of the claim, we begin by showing that 
$z \in R(A)$. For that, take $w_n \in H^1(\Tt^d)$ satisfying $z_n=Aw_n$. From the preceding claim, 
it follows that $\{w_n\}_{n\in \Nn}$ is a Cauchy sequence converging to some
$w\in H^1(\Tt^d)$. By the continuity of $A$, we have $z=Aw$. Thus, $z \in R(A)$.

Next, to establish the last part of the claim, 
suppose that $R(A) \neq H^1(\Tt^d)$. In this case, there exists a non-zero vector,  $v
\in R(A)^\perp$. Then, we get
\begin{align*}
0&=(Av,v)=B_\lambda[v,v]=\int_{\Tt^d} m_\lambda |Dv|^2+(g^{-1})'(g(m_\lambda))\{
\epsi v +Du_\lambda \cdot Dv\}^2\, \dx.
\end{align*}
This contradicts $v \neq 0$.

\begin{claim}
\eqref{LP} has a unique solution $\phi \in C^{2,\alpha}(\Tt^d)$.
\end{claim}
To prove this last claim, we define a bounded linear functional, $T:H^1(\Tt^d) \to \Rr,$ by 
\begin{equation*}
T(w)=\int_{\Tt^d} \xi w \, \dx.
\end{equation*}
The Riesz Representation Theorem guarantees that there exists a unique
$\hat{w} \in H^1(\Tt^d)$ satisfying $T(w)=(\hat{w},w) $ for all $w \in H^1(\Tt^d)$.
Taking $\phi=A^{-1}\hat{w}$, we get
\begin{equation*}
B_\lambda[\phi,w]=(A\phi,w)_E=(\hat{w},w)_E=T(w);
\end{equation*}
that is, $\phi \in H^1(\Tt^d)$ is the unique weak solution of \eqref{LP}.
Because \eqref{LP} is a uniformly elliptic equation and the coefficients belong
to $C^{0,\alpha}(\Tt^d)$, we conclude that $\phi \in C^{2,\alpha}(\Tt^d)$. 
\end{proof}

To finish the proof of Theorem \ref{result1} we verify that $\Lambda$ is relatively open. 
This is achieved in the next proposition. 

\begin{pro}
Consider the setting of Problem \ref{P1} and suppose that Assumptions \ref{osc} and \ref{beta} hold.
Let $\Lambda$ as in \eqref{lamb}. Then, $\Lambda$ is relatively open in
$[0,1]$. \end{pro}
\begin{proof}
 By the preceding lemma, Lemma \ref{isolem}, we can apply the implicit function theorem (see \cite{Die}) to the operator $J$ to conclude that  $\Lambda$ is open. Therefore, for any $\lambda \in \Lambda$, there exists $\delta>0$ such that for any $\hat{\lambda} \in ({\lambda}-\delta,{\lambda}+\delta)$, we can find $\hat{u} \in C^{2,\alpha}(\Tt^d)$ such that $J(\hat u,\hat \lambda)=0$.
\end{proof}

By combining the previous results, we prove Theorem \ref{result1} as follows. 

\begin{proof}[Proof of Theorem \ref{result1}]
Since $1\in \Lambda$, there exists a classical solution $u^\epsi$ for \eqref{EL}. 
 Take $m^\epsi=f^{-1}(\epsi u^\epsi+\frac{1}{2}|Du^\epsi|^2+V)$. Then,  $(u^\epsi,m^\epsi)$ solves \eqref{DP}. 
 The identity \eqref{identity} gives that \eqref{DP} has a unique classical solution.
\end{proof}
 
 Finally, we show that
 under Assumptions \ref{osc} and \ref{beta}, we have the convergence of the solutions of \eqref{DP}. 
%
 
 \begin{proof}[Proof of Corollary \ref{cor1}]
The estimates in Section \ref{esti} and \ref{esti2} do not depend on $\epsi$.
Therefore, we can extract a subsequence $\epsi_j$ such that 
$\epsi_j u^{\epsi_j}$ converges uniformly to a constant $-\bar{H}$ and 
{$(u^{\epsi_j}-\int_{\Tt^d}u^{\epsi_j} dx ,m^{\epsi_j})$} converges to some $(u,m)$ in $C^{2, \alpha}\times C^{1, \alpha}$. 
Therefore,  $(u,m,\bar{H})$ solves \eqref{CP}. 
By the results in Section \ref{LU}, $m$ and $\bar{H}$ are uniquely determined.
Accordingly, the limit of $\epsi_j u^{\epsi_j}$  and $m^{\epsi_j}$ does not depend on
the subsequence. 
Because of Proposition \ref{lbd}, we have that $m>0$. Thus, there exists a unique solution, $(u,m,\bar{H})$ of \eqref{CP} satisfying the additional condition
{$\int_{\Tt^d} u dx=0$}. 
%
%
%
%
\end{proof}


\section{Refined asymptotics}
\label{RA}

Now, we investigate the asymptotic behavior of $\{u^\epsi - {\Hh}/\epsi\}_{\epsi>0}$ and prove
Theorem \ref{TT}, thus improving the converge results in 
Corollary \ref{cor1}.
%


First, to address Problem \ref{P4}, we consider the linearized discounted problem that we state now. 
\begin{problem}
	\label{P5}
Let $g$ be as in Problem \ref{P1} with $g\in C^\infty$ and let $(u,m)\in C^\infty (\Tt^d)\times C^\infty(\Tt^d)$ solve Problem \ref{P2} with $m>0$. Suppose that $\epsi>0$. Given $A,B \in C^\infty(\Tt^d)$, find $v^\epsi, \theta^\epsi: \Tt^d \to \Rr$ such that
\begin{equation}\label{LDP}
\begin{cases}
&\epsi v^\epsi +u+Du\cdot Dv^\epsi=g'(m)\theta^\epsi +A\quad\rm{in}\  \Tt^d, \\
& \epsi \theta^\epsi-\mathrm{div}(m Dv^\epsi)-\mathrm{div}(\theta^\epsi Du) = 1-m +\mathrm{div}(B)\quad \rm{in} \ \Tt^d. \\
\end{cases}
\end{equation}
\end{problem}


\begin{proposition}	
Suppose Assumption \ref{beta} holds. 
Then,
Problem \ref{P5} has a unique weak solution $(v^\epsi,\theta^\epsi)\in H^1(\Tt^d)\times L^2(\Tt^d)$.	
\end{proposition}

\begin{proof}
	Because $m>0$, Assumption \ref{beta} 
(or  the alternative assumption in Remark \ref{lbrem})	
	gives that $g'(m)$ is bounded by below.  
	From the first equation in \eqref{LDP}, we get 
	\begin{equation}
	\label{62B}
		 \theta^\epsi=\frac{\epsi v^\epsi+u+Du\cdot Dv^\epsi-A}{g'(m)}.
   \end{equation} 
	Using the previous expression for $\theta^\epsi$ in the second equation in \eqref{LDP}, we obtain
	\begin{align}\label{bbb}
	&-\mathrm{div}\left(mDv^\epsi+\frac{\epsi v^\epsi+Du\cdot Dv^\epsi}{g'(m)}Du \right)+\frac{\epsi(\epsi v^\epsi+Du\cdot Dv^\epsi)}{g'(m)}\\\notag &\qquad=1-m+\div(B)+\mathrm{div}\left(\frac{u-A}{g'(m)}Du\right)-\frac{\epsi(u-A)}{g'(m)}.
	\end{align}
	Therefore, it suffices to show that \eqref{bbb} has a weak solution. For that, we define a bilinear form,  $K:H^1(\Tt^d)\times H^1(\Tt^d)\to \Rr$, by
	\begin{equation*}
	K[\phi_1,\phi_2]=\int_{\Tt^d} mD\phi_1 \cdot D\phi_2+\frac{\epsi \phi_1+Du\cdot D\phi_1}{g'(m)}Du\cdot D\phi_2+\frac{\epsi \phi_2(\epsi \phi_1+Du\cdot D\phi_1)}{g'(m)} dx.
	\end{equation*}
	Because $m$ and $u$ are smooth with $g'(m)$ bounded by below, we see that  $K$ is a bounded bilinear form. Moreover, for all $\phi \in H^1(\Tt^d)$,
	\begin{equation*}
	K[\phi,\phi]=\int_{\Tt^d} m|D \phi|^2+\frac{(Du\cdot D\phi+\epsi \phi)^2}{g'(m)} dx.
	\end{equation*}
	Hence, $K$ is coercive. Thus, applying the Lax-Milgram theorem, we see that \eqref{bbb} has a unique weak solution, $v^\epsilon\in H^1(\Tt^d)$. Then, using \eqref{62B} and taking into account that $g'(m)$ is bounded by below, we obtain a weak solution $\theta^\epsi\in L^2(\Tt^d)$. 
\end{proof}

\begin{proposition}
Let
$(v^\epsi,\theta^\epsi)\in H^1(\Tt^d)\times L^2(\Tt^d)$ be a weak solution of Problem \ref{P5}. 
	Then, there exists a constant $C>0$ independent of $\epsi$ such that
	\[ \|\epsi v^\epsi\|_{L^2(\Tt^d)}+\| \theta^\epsi \|_{L^2(\Tt^d)}+\|Dv^\epsi \|_{L^2(\Tt^d)} \leq C(\|A\|_{L^2(\Tt^d)}+\|B\|_{L^2(\Tt^d)}+1).\]
\end{proposition}

\begin{proof}
	We multiply the first equation in \eqref{LDP} by $\theta^\epsi$ and the second one by $v^\epsi$. Next, we subtract the resulting expressions to get	
	\[u\theta^\epsi+\theta^\epsi Du\cdot Dv^\epsi+v^\epsi \mathrm{div}(mDv^\epsi)+v^\epsi \mathrm{div}(\theta^\epsi Du)=g'(m)|\theta^\epsi|^2-(1-m)v^\epsi-v^\epsi \mathrm{div}(B)+A\theta^\epsi. \]
	Integrating by parts, we obtain
	\[ \int_{\Tt^d} g'(m)|\theta^\epsi|^2+m|Dv^\epsi|^2 \mbox{ }dx= \int_{\Tt^d} (u-A)\theta^\epsi +(1-m)v^\epsi+v^\epsi \mathrm{div}(B) \mbox{ }dx.\]	
Using Poincar\'e's inequality, we conclude that
	\begin{align*}
	\int_{\Tt^d} m|Dv^\epsi|^2 \mbox{ }dx
	&\leq  \int_{\Tt^d} (u-A)\theta^\epsi +(1-m)v^\epsi+v^\epsi \mathrm{div}(B) \mbox{ }dx \\
	&= \int_{\Tt^d} (u-A)\theta^\epsi +(1-m)\left(v^\epsi-\int v^\epsi \mbox{ }dx \right)+v^\epsi \mathrm{div}(B)\mbox{ }dx \\
	&\leq \|u\|_{L^2}\|\theta^\epsi \|_{L^2}+\|A\|_{L^2}\|\theta^\epsi \|_{L^2}+\|1-m\|_{L^2}\|Dv^\epsi\|_{L^2}+\|B\|_{L^2}\|Dv^\epsi\|_{L^2}.
	\end{align*}
	Hence, taking into account that $m$ is bounded by below, 
	\begin{equation}\label{Dv}
	\|Dv^\epsi\|^2_{L^2}\leq C(\|\theta^\epsi\|_{L^2}+\|A\|_{L^2}\|\theta^\epsi\|_{L^2}+\|B\|^2_{L^2}).
	\end{equation}
Arguing analogously, we obtain
	\begin{align*}
	\int_{\Tt^d} g'(m)|\theta^\epsi|^2 \mbox{ }dx
	\leq \|u\|_{L^2}\|\theta^\epsi \|_{L^2}+\|A\|_{L^2}\|\theta^\epsi \|_{L^2}+\|1-m\|_{L^2}\|Dv^\epsi\|_{L^2}+\|B\|_{L^2}\|Dv^\epsi\|_{L^2}.
	\end{align*}
	Hence, by \eqref{Dv},
	\[ \|\theta^\epsi\|^2_{L^2}\leq C(\|A\|^2_{L^2}+\|B\|^2_{L^2}+1).\]
	Combining the preceding inequality with \eqref{Dv},  we have the estimate
	\[ \|Dv^\epsi\|^2_{L^2}\leq C(\|A\|^2_{L^2}+\|B\|^2_{L^2}+1).\]
	Finally, the first equation in \eqref{LDP} yields
	\[ \|\epsi v^\epsi \|^2_{L^2}\leq C(\|A\|^2_{L^2}+\|B\|^2_{L^2}+1).
	\qedhere
	\]
\end{proof}

Next, we bootstrap higher regularity for $(v^\epsi, \theta^\epsi)$. 

\begin{proposition}
Let
$(v^\epsi,\theta^\epsi)$ be a weak solution of Problem \ref{P5}. 	Fix $h\in \{1,2,,...,d\}$ and let $z=v^\epsi_{x_h}$. Then, for each $k\in \Nn$, there exists a constant $C_k>0$ such that 
\[ \|z\|_{H^k(\Tt^d)} \leq C_k(1+\|A\|_{H^k(\Tt^d)}+\|B\|_{H^k(\Tt^d)}).\] 
\end{proposition}
\begin{proof}
We begin by rewriting \eqref{bbb} as
\begin{align*}
-\div\left(mDv^\epsi+\frac{Du\cdot Dv^\epsi}{g'(m)}Du \right)=&-\div\left(\frac{\epsi v^\epsi}{g'(m)}Du\right)-\frac{\epsi(\epsi v^\epsi+Du\cdot Dv^\epsi)}{g'(m)}+1-m\\&+\div(B)+\div\left(\frac{u-A}{g'(m)}Du\right)-\frac{\epsi(u-A)}{g'(m)}.
\end{align*}
Next, we fix $h \in \{1,2,...,d\}$ and let $z=v^\epsi_{x_h}$. Differentiating the preceding equation with respect to $x_h$, we obtain 
\begin{equation}\label{kkk}
(a^{ij}z_{x_j})_{x_i}=\phi_{x_h}+\psi^i_{x_i},
\end{equation}
where
\begin{equation*}
a^{ij}=\delta_{ij}m+\frac{u_{x_i}u_{x_j}}{g'(m)},
\end{equation*}
\begin{equation*}
\psi^i=-m_{x_h}v_{x_i} -\frac{Du\cdot Dv^\epsi}{g'(m)}  u_{x_i x_h}+\frac{Du\cdot Dv^\epsi}{g'(m)^2} g''(m)u_{x_i}m_{x_h}+\frac{Dv^\epsi \cdot Du_{x_h}}{g'(m)}u_{x_i} ,
\end{equation*}
and
\begin{align*}
\phi=&-\div\left(\frac{\epsi v^\epsi}{g'(m)}Du\right)-\frac{\epsi(\epsi v^\epsi+Du\cdot Dv^\epsi)}{g'(m)}+1-m\\&+\div(B)+\div\left(\frac{u-A}{g'(m)}Du\right)-\frac{\epsi(u-A)}{g'(m)}.
\end{align*}
By the previous proposition, we know that
\[ \|z\|_{L^2(\Tt^d)}\leq C(\|A\|_{L^2(\Tt^d)}+\|B\|_{L^2(\Tt^d)}+1).\]
Furthermore,
we have the estimates
\[\| \phi\|_{L^{2}(\Tt^d)}\leq C(\|A\|_{H^1(\Tt^d)}+\|B\|_{H^1(\Tt^d)}+1),\]
and
\[\|\psi\|_{L^2(\Tt^d)}\leq C(\|A\|_{L^2}+\|B\|_{L^2}+1).\]
	
Let $k\geq 0$. Multiplying \eqref{kkk} by $z$ and integrating by parts, we get
	\begin{equation*}
	\int_{\Tt^d} a^{ij}z_{x_j}z_{x_i} dx=\int_{\Tt^d} -\phi z_{x_h}+ -h z_{x_i}dx.
	\end{equation*}
	Because of the uniform ellipticity of $a^{ij}$, we get
	\[ \int_{\Tt^d} |Dz|^2 dx \leq C (\|\phi\|_{L^2}+\|h\|_{L^2})\|Dz\|_{L^2}. \]
	Hence,
	\[ \|Dz\|_{L^2(\Tt^d)} \leq C(\|\phi \|_{L^2(\Tt^d)}+\|h\|_{L^2(\Tt^d)})\leq C(1+\|A\|_{H^1(\Tt^d)}+\|B\|_{H^1(\Tt^d)}).\]
	Therefore,
	\[\| \phi\|_{H^{1}(\Tt^d)}\leq C(\|A\|_{H^2(\Tt^d)}+\|B\|_{H^2(\Tt^d)}+1),\]
	and
	\[\|\psi\|_{H^1(\Tt^d)}\leq C(\|A\|_{H^1}+\|B\|_{H^1}+1).\]
The conclusion follows by iterating this argument for higher derivatives. 
\end{proof}

\begin{proposition}
Let
$(v^\epsi,\theta^\epsi)$ be a weak solution of Problem \ref{P5}. 	
	Then, for each $k\in \Nn$, there exists a constant $C_k>0$ such that 
	\[ \|\epsi v^\epsi\|_{H^k(\Tt^d)}+\|\theta^\epsi \|_{H^k(\Tt^d)} \leq C_k(1+\|A\|_{H^k(\Tt^d)}+\|B\|_{H^k(\Tt^d)}).\]
	In particular, $(v^\epsi,\theta^\epsi)$ is a classical solution of \eqref{LDP}. 
\end{proposition}
\begin{proof}
Differentiating the first equation in \eqref{LDP}, we get
	\[\|D\theta\|_{L^2(\Tt^d)}\leq C(1+\|A\|_{H^1(\Tt^d)}+\|B\|_{H^1(\Tt^d)}).\]
	The above implies
	\[ \|\epsi v^\epsi \|_{H^1(\Tt^d)}\leq  C(1+\|A\|_{H^1(\Tt^d)}+\|B\|_{H^1(\Tt^d)}).\]
	Iterating the preceding steps, we get the result.
\end{proof}

\begin{proposition}\label{lambda}
For $\epsilon>0$, let
$(v^\epsi,\theta^\epsi)$
be a weak solution of 
 Problem \ref{P5} and assume that
	\[ \|A\|_{L^\infty}+\|B\|_{L^\infty}\leq C\epsi  \]
	for some constant $C>0$. Then, there exists a solution $(v,\theta,\lambda)$ of Problem \ref{P4} and, 
	for each $k\in \Nn$,
	\[\lim_{\epsi \to0}\left( \|\epsi v^\epsi-\lambda \|_{L^\infty(\Tt^d)}+\|\theta^\epsi -\theta\|_{H^k(\Tt^d)}+\left\|\left(v^\epsi-\int_{\Tt^d} v^\epsi \mbox{ }dx \right)-v \right\|_{H^k(\Tt^d)}\right)=0.\]
\end{proposition}
\begin{proof}
	By the previous estimates on the solutions of Problem \ref{P5}, we can choose a subsequence such that $\epsi v^\epsi \to -\lambda$, $\theta^\epsi \to \theta$ and $v^\epsi -\int v^\epsi \to v$. Clearly, $(v,\theta, \lambda)$ solves \eqref{LLP}. Because the solution to \eqref{LLP} is unique, the limit is independent
	of the particular sequence. Therefore,  $(v^\epsi,\theta^\epsi)$ converges to $(v, \theta)$.
\end{proof}


Finally, we present the proof of Theorem \ref{TT}. 

\begin{proof}[Proof of Theorem \ref{TT}]
	Fix $k\in \Nn$ and set
	\[E_k=\{(v,\theta)\in H^{k+1}(\Tt^d)\times H^{k}(\Tt^d), \|\epsi v\|_{H^k (\Tt^d)}+\|Dv\|_{H^{k}(\Tt^d)}+\|\theta \|_{H^{k}(\Tt^d)} \leq \hat{C_k}\},    \]
	where $\hat{C_k}$ is to be chosen later. For $(v,\theta)\in E_k$, 
	we find
	 $(\hat{v},\hat{\theta})$ solving \eqref{LDP}, where
	\[A(x)= \frac{-\epsi^2 v-\frac{\epsi^2}{2}|Dv|^2+g(m+\epsi \theta)-g(m)-\epsi g'(m)\theta}{\epsi} ,\]
	and
	\[B(x)=\frac{\epsi^2\theta Dv}{\epsi}   .\]
	Because,
	\[ \|A\|_{H^k(\Tt^d)}+\|B\|_{H^k(\Tt^d)} \leq C \hat{C_k}^2 \epsi, \]
	we obtain
	\begin{align*}
	\|\epsi v\|_{H^k (\Tt^d)}+\|Dv\|_{H^k(\Tt^d)}+\|\theta \|_{H^k(\Tt^d)}
	&\leq C_k(\|A\|_{H^k(\Tt^d)}+\|B\|_{H^k(\Tt^d)}+1)\\
	&\leq C_k(1+\hat{C_k}^2\epsi).
	\end{align*}
	We can choose $\hat{C_k}$ such that, for $\epsi$ small enough, the right-hand side is less than $\hat{C_k}$. Then, it holds that $(v,\theta) \to (\hat{v},\hat{\theta})$ has a fixed point $(v^\epsi, \theta^\epsi)$. We remark that $(\epsi v^\epsi+u+\frac{\bar{H}}{\epsi},m+\epsi \theta^\epsi)$ solves \eqref{DP} and therefore it is equal to $(u^\epsi,m^\epsi)$. Hence, by the previous proposition, for suitably large $k$, as $\epsi \to 0$,
	\[\|u^\epsi-\frac{\bar{H}}{\epsi}-u-\lambda\|_{\infty}=\|\epsi v^\epsi-\lambda\|_{\infty} \to 0.\qedhere\]
\end{proof}

\section{Convergence and selection}\label{selection}
Now, we investigate the behavior of $(u^\epsi,m^\epsi)$ as $\epsi \to 0$
in the case where Problem \ref{P2} may have multiple solutions; that is,  
when  Assumptions \ref{osc} and \ref{beta} do not hold.
We are interested in which of the solutions
of Problem \ref{P2} arise as a limit of solutions
of Problem \ref{P1}. 
Without Assumptions \ref{osc} and \ref{beta}, smooth solutions may not exist. 
Therefore, we need to work with weak solutions.
For Problem \ref{P1}, weak solutions were shown to exist in \cite{FG2}. 
In Section \ref{WSS}, we review those existence results 
and use them to show the existence of a solution for Problem \ref{P2}.  
Then, in Section \ref{SSMM},  we construct certain measures on phase space that generalize Mather measures.
Next, in Section \ref{SSS}, we prove our main selection result, Theorem \ref{result2}. We end the paper with a discussion of
an explicit example, in Section \ref{SSEA}.

\subsection{Regular weak solutions}
\label{WSS}

%
%

We begin this section by proving the following result on the stability of regular weak solutions. In particular, 
since the estimate of regular weak solutions of Problem \ref{P1} was proved in \cite{FG2}, we obtain the existence
of regular weak solutions for Problem \ref{P2}.  

\begin{pro}
\label{stabilitypro}
	Suppose that Assumption \ref{H3} holds. 
	Let $(m^\epsilon, u^\epsilon)$ be a regular weak solution of Problem \ref{P1}. 	
	Assume that 
	$m^\epsilon\rightharpoonup m$ weakly in $L^1(\Tt^d)$,  $u^\epsilon -\int u^\epsi dx \rightharpoonup u$ weakly in $W^{1,2}(\Tt^d)$,
	and that $\epsilon \int u^\epsilon \mbox{ }dx\to -\Hh$. Then, 
	$(m, u, \Hh)$ is a regular weak solution of Problem \ref{P2}.
\end{pro}
\begin{proof}
	Properties 1 and 3 in Theorem \ref{ws} are immediate; that is, 
	\[
	m\geq 0, \qquad \int_{\Tt^d} m dx=1, \qquad \|{ u^\epsi -\int_{\Tt^d} u^\epsi dx}\|_{W^{1,2}(\Tt^d)}\leq C. 
	\]
	
	From Property 2, 
	we conclude that, through a subequence $(m^\epsilon)^{\frac{\alpha+1}{2}}$, converges weakly
	in $W^{1,2}$ to a function $\eta(x)$. Moreover,  by Rellich-Kondrachov theorem, 
	$(m^\epsilon)^{\frac{\alpha+1}{2}}\to \eta$ in $L^2$ and extracting a further sequence if necessary 
	also almost everywhere. Therefore, $m^\epsilon\to \eta^{\frac{2}{\alpha+1}}=m$ almost everywhere. 
	Accordingly $(m^\epsilon)^{\frac{\alpha+1}{2}}$ converges to $m^{\frac{\alpha+1}{2}}$ weakly in $W^{1,2}$, strongly in $L^2$ and almost everywhere. Consequently
	\[
	\|m^{\frac{\alpha+1}{2}}\|_{W^{1,2}(\Tt^d)}\leq C.
	\]
	Next, we examine Property 5 in Theorem \ref{ws}. By compactness,  $(m^\epsilon)^{\frac{\alpha+1} 2}Du^\epsilon$
	converges weakly in BV, through a subsequence, to a function $\psi\in BV$ with
	$\|\psi\|_{BV(\Tt^d)}\leq C$. 
	Because $(m^\epsilon)^{\frac{\alpha+1} 2}$ converges to $m^{\frac{\alpha+1} 2}$ strongly in $L^2$ and $Du^\epsilon$ converges weakly to $Du$ in $L^2$, we have for any 
	$\varphi\in C^\infty(\Tt^d)$ 
	\[
	\int_{\Tt^d} \varphi (m^\epsilon)^{\frac{\alpha+1} 2}Du^\epsilon dx 
	\to 
	\int_{\Tt^d} \varphi  m^{\frac{\alpha+1} 2}Du dx. 
	\]
	Therefore, 
	\[
	\psi=m^{\frac{\alpha+1} 2}Du. 
	\]
	
	Finally, we address the limit properties corresponding to 
	\eqref{subsol}, \eqref{intweaksol} and \eqref{epweaksol}.
	We take a smooth function, $\varphi\in C^\infty(\Tt^d)$, multiply \eqref{epweaksol} by $\varphi$ and integrate. 
	Because
	\[
	\int_{\Tt^d} \epsilon (m^\epsilon-1)\varphi dx \to 0, 
	\]
	we have
	\[
	\int_{\Tt^d} m^\epsilon D \varphi Du^\epsilon dx \to 0.
	\]
Because of Rellich-Kondrachov theorem
\[
\int_{\Tt^d} (m^\epsilon)^{q\frac{\alpha+1} 2} { dx}\to \int_{\Tt^d} m^{q\frac{\alpha+1} 2} { dx}, 
\]
for any $q<2^*$. In particular,
for $\alpha$ in the range of Assumption \ref{H3}
this implies $m^\epsilon\to m$ strongly in $L^2$. Using the weak convergence of $u^\epsilon$
in $W^{1,2}$ we conclude that 
\[
\int_{\Tt^d} m D \varphi Du dx =0.
\]	
	Next, we select a smooth non-negative function, $\varphi\geq 0$, multiply 
	\eqref{subsol} by $\varphi$, and integrate in $\Tt^d$. We have
	\[
	\int_{\Tt^d}  \left(\epsilon u^\epsilon - V(x)\right) \varphi dx \to \int_{\Tt^d} \left(-\Hh -V(x) \right) \varphi dx. 
	\]
	Moreover, by convexity 
	\[
	\liminf_{\epsilon \to 0 }\int_{\Tt^d} \frac{|Du^\epsilon |^2}{2} \varphi dx
	\geq \int_{\Tt^d} \frac{|Du |^2}{2} \varphi dx.
	\]
Finally, we observe that
\[
\int_{\Tt^d} g(m^\epsilon)-g(m) { dx}
=\int_{\Tt^d} \int_0^1 g'(s m^\epsilon+(1-s) m) (m^\epsilon-m){dx}.
\]
For any $\alpha>0$, we can select
$p$ and $p'$ such that $\frac 1 p+\frac 1 {p'}+1$, 
 $p (\alpha-1)<\frac{2^*}{2} (\alpha+1)$, and  $p'<\frac{2^*}{2} (\alpha+1)$. 
 Next, we estimate
\[
\|g'(s m^\epsilon+(1-s) m)\|_{L^p}\leq C (\|m^\epsilon\|_{L^{p(\alpha-1)}}+\|m\|_{L^{p(\alpha-1)}})\leq C. 
\]
Therefore, 
since $\|m^\epsilon-m\|_{L^{p'}}\to 0$, we conclude that 
\[
\int_{\Tt^d} g(m^\epsilon)-g(m) {dx}\to 0. \qedhere
\]	
\end{proof}

\begin{pro}
Suppose that Assumption \ref{H3} holds. 
Let $(m^\epsilon, u^\epsilon)$ be a regular weak solution of Problem \ref{P1}. Then, 
there exists a constant $C>0$ independent of $\epsilon$ such that  
\begin{equation}
\label{EZZ}
\int_{\Tt^d} |Du^\epsilon|^2 m^\epsilon +m^\epsilon g(m^\epsilon) dx\leq C. 
\end{equation} 
\end{pro}
\begin{proof}
Because of properties 1 and 2 in Theorem \ref{ws}, Assumption \ref{H3} implies that  
\[
\int_{\Tt^d} m^\epsilon g(m^\epsilon) dx\leq C. 
\]
Then, \eqref{EZZ} follows by integrating \eqref{intweaksol}.
\end{proof}

\begin{remark}
If Assumptions \ref{H3}  holds, a similar estimate holds for regular weak solutions
$(m, u, \Hh)$ of Problem \ref{P2}. Namely,
there exists a constant $C>0$ such that  
\[
\int_{\Tt^d} |Du|^2 m  +m  g(m) dx\leq C. 
\] 
\end{remark}

%
%

\subsection{Mather measures}
\label{SSMM}

We begin by introducing a class of phase-space
probability measures called Mather measures, see \cite{Mane2} and \cite{Ma}.  
These measures were introduced in the context of Lagrangian mechanics and later used
to  examine the properties of Hamilton--Jacobi equations in \cite{FATH1, FATH2, FATH3, FATH4}
and in \cite{EGom1, EGom2}. In the context of the selection problem, 
generalized Mather measures were first used in \cite{MR2458239}.
As previously, we suppose that Assumption \ref{H3} holds. Accordingly, 
we work with regular weak solutions of Problems \ref{P1} and \ref{P2}. 

Fix a regular weak solution $(u^\epsi,m^\epsi)$ of Problem \ref{P1} and a regular weak solution 
$(u, m)$ of Problem \ref{P2}. 
Next, we rewrite \eqref{DP} as
\begin{equation}
\label{DP2}
        \begin{cases}
                &\epsi u^\epsi+\frac{1}{2}|Du^\epsi|^2 +W^\epsi(x)=0 \quad\rm{in}\  \Tt^d, \\
                &\epsi m^\epsi -\mathrm{div}(m^\epsi Du^\epsi) = \epsi \quad \rm{in} \ \Tt^d, 
        \end{cases}
\end{equation}
where $W^\epsi(x)=V(x)-g(m^\epsi)$ and, assuming without loss of generality 
that $\Hh=0$, we rewrite \eqref{CP} as
\begin{equation}
\label{CP2}
        \begin{cases}
                &\frac{1}{2}|Du|^2+W(x)=0 \quad\rm{in}\  \Tt^d, \\
                & -\mathrm{div}(mDu) = 0 \quad \rm{in} \ \Tt^d, 
        \end{cases}
\end{equation}
where $W(x)=V(x)-g(m)$.

\begin{pro}
\label{P74}
Suppose that Assumption \ref{H3} holds.
Let $(u^\epsilon, m^\epsilon)$ be a regular weak solution of Problem \ref{P1}. 
Let $L^\epsi=\frac{1}{2}|v|^2-W^\epsi(x)$ with $W^\epsi(x)=V(x)-g(m^\epsi)$. 
Consider the phase-space measure, $\mu^\epsi$, the $\epsi$-Mather measure,
determined by
\begin{equation}
\label{emm}
\int_{\Tt^d\times \Rr^d} \phi(x,v)\mbox{ d}\mu^\epsi=\int_{\Tt^d}\phi(x,-Du^\epsi)m^\epsi \,\dx
\end{equation}
for all $\phi\in C(\Tt^d\times \Rr^d)$. 
Then, $\mu^\epsi$
satisfies the discounted holonomy condition
\begin{equation}
\label{dhc}
\int_{\Tt^d\times \Rr^d}-\epsi \varphi(x)+vD_x \varphi(x)\mbox{ d}\mu^\epsi=-\epsi \int_{\Tt^d}\varphi \,\d x
\end{equation}
for all $\varphi\in C^1(\Tt^d)$. 
Moreover, we have
\begin{equation}\label{CCC}
\int_{\Tt^d\times \Rr^d} L^\epsi(x,v)\mbox{ d}\mu^\epsi
=\epsi \int_{\Tt^d} u^\epsi \,\dx. 
\end{equation}
\end{pro}
\begin{proof}
Because \eqref{epweaksol} holds in the sense of distributions, 
the discounted holonomy condition for 
$\mu^\epsi$, \eqref{dhc}, follows immediately.
Next, recall that if $m^\epsilon$ is an integrable non-negative function
then the space, $L^2_{1+m^\epsilon}(\Tt^d)$, of all measurable functions $f:\Tt^d\to \Rr$ that satisfy
\[
\int_{\Tt^d}  |f|^2 (1+m^\epsilon) dx<\infty
\]
is a Hilbert space. Moreover, if $\eta_\delta$ is a standard mollifier, 
we have
\[
\eta_\delta * f\to f
\]
in $L^2_{1+m^\epsilon}(\Tt^d)$. Due to \eqref{subsol} and \eqref{intweaksol}, 
we have
\[
Du^\epsilon\in L^2_{1+m^\epsilon}(\Tt^d).
\]
Accordingly, because of \eqref{dhc}, we have
\[
\int_{\Tt^d\times \Rr^d}-\epsi (\eta_\delta *u^\epsilon)+vD_x (\eta_\delta *u^\epsilon) \mbox{ d}\mu^\epsi=-\epsi \int_{\Tt^d}(\eta_\delta *u^\epsilon)\,\d x.
\]
Taking the limit $\delta\to 0$, we obtain
\begin{equation}
\label{3X}
\int_{\Tt^d\times \Rr^d}-\epsi u^\epsilon+vD_x u^\epsilon\mbox{ d}\mu^\epsi=-\epsi \int_{\Tt^d}u^\epsilon\,\d x.
\end{equation}
Taking into account the definition of
$L^\epsilon$ and using the identities \eqref{intweaksol} and \ref{3X}, we conclude that 
\begin{align}\label{CCC2}
&\int_{\Tt^d\times \Rr^d} L^\epsi(x,v)\mbox{ d}\mu^\epsi
=\int_{\Tt^d\times \Rr^d} \frac{1}{2}|v|^2-W^\epsi(x) \mbox{ d}\mu^\epsi \\ \notag
&=\int_{\Tt^d} \left(-\frac{1}{2}|Du^\epsi|^2-W^\epsi(x)\right)m^\epsi \,\dx+\int_{\Tt^d\times \Rr^d} vD_xu^\epsi-\epsi u^\epsi \mbox{ d}\mu^\epsi+\epsi \int_{\Tt^d} u^\epsi \,\dx \\ \notag
&=\epsi \int_{\Tt^d} u^\epsi \,\dx. 
\end{align}

\end{proof}

\begin{remark}
From \eqref{dhc} and \eqref{CCC}, we conclude that 
$\mu^\epsi$ is a discounted Mather measure with trace $\d x$ in the sense of the definition in \cite{MR2458239}.
\end{remark}

Similarly, for Problem \ref{P2}, we have the following result. 

\begin{pro}
\label{P76}
	Suppose that Assumption \ref{H3} holds.
	Let $(u, m, \Hh)$ be a regular weak solution of Problem \ref{P2}. Assume without loss
	of generality that $\Hh=0$.  
	Let $L=\frac{1}{2}|v|^2-W$ with $W(x)=V(x)-g(m)$. 
	Consider the phase-space measure, $\mu$, the Mather measure,
	determined by
\begin{equation*}
\int_{\Tt^d\times \Rr^d} \phi(x,v)\mbox{ d}\mu=\int_{\Tt^d} \phi(x,-Du)m\,\dx,
\end{equation*}
for all $\phi\in C(\Tt^d\times \Rr^d)$. 
Then, $\mu$ satisfies
the holonomy constraint
\begin{equation}
\label{hc}
\int_{\Tt^d\times \Rr^d} vD_x\varphi(x)\mbox{ d}\mu=0 
\end{equation}
for all $\varphi\in C^1(\Tt^d)$. 
Moreover
\begin{equation*}
\int_{\Tt^d \times \Rr^d} L(x,v)\mbox{ d}\mu=\int_{\Tt^d} \left(\frac{1}{2}|Du|^2-W(x)\right)m\,\dx=0. 
\end{equation*} 
\end{pro}
\begin{proof}
The proof is analogous to the one of Proposition \ref{P74}
\end{proof}



\subsection{Selection}
\label{SSS}

The goal of this section is to prove Theorem \ref{result2} and, hence, 
establish a selection criterion for the limit of $u^\epsi$ and prove the convergence of $m^\epsilon$. 
Our proof is inspired by the one in \cite{MR2458239}. 
%
%
%
%
%
%
%
%
%
%
%

\begin{proof}[Proof of Theorem \ref{result2}]
Let  $(u^\epsi,m^\epsi)$ be a regular weak solution of Problem \ref{P1}.
Suppose that $u^\epsi -\int u^\epsi dx \rightharpoonup \bar{u}$ in $H^1(\Tt^d)$ and that 
$m^\epsi\rightharpoonup \bar{m}$ weakly in $L^{1}$. Note that due to the bounds
Theorem \ref{ws}, we have that $m^\epsi\to \bar m$ strongly in $L^p$ 
 for $p<\frac{2^*}{2}(\alpha+1)$. 
Let $(u,m)$ be a regular weak solution of Problem \ref{P2}. 
Let $\mu^\epsi$ and $\mu$ be the Mather measure constructed in the previous section in 
Propositions \ref{P74} and \ref{P76}. 

For any $v \in \Rr^d$ and almost every $x \in \Tt^d$, we have
\begin{align*}
-v\cdot p-L^\epsi (x,v)\leq \frac{1}{2}|p|^2+W^\epsi (x), 
\end{align*}
where $W^\epsi$ is as in Proposition \ref{P74}. 
Consider a standard mollifier $\eta_\delta$, and let $p=D(\eta_\delta*u)$.  
Then,
 for  $v \in \Rr^d$ and almost every $x \in \Tt^d$,
\begin{align}
\label{A0}
-v\cdot D(\eta_\delta*u)-L^\epsi(x,v)+W-W^\epsi &\leq \frac{1}{2}|D(\eta_\delta*u)|^2+W^\epsi+W-W^\epsi\\&\notag =
\frac{1}{2}|D(\eta_\delta*u)|^2+W,
\end{align}
where $W$ is as in Proposition \ref{P76}. 
Integrating the left-hand side of the preceding expression with respect to $\mu^\epsi$  and using the holonomy condition,
\eqref{dhc}, and \eqref{CCC}, we obtain 
\begin{align}
\label{A1}
&\int_{\Tt^d\times \Rr^d} 
v\cdot D(\eta_\delta*u)-L^\epsi(x,v)+W-W^\epsi  d\mu^\epsilon\\
\notag &\quad =
-\epsi \int_{\Tt^d} (\eta_\delta*u) m^\epsi dx+\epsi \int_{\Tt^d} (\eta_\delta*u) dx-\epsi \int_{\Tt^d} u^\epsi dx +\int_{\Tt^d} (W-W^\epsi)m^\epsi dx\\
\notag &\quad =
-\epsi \int_{\Tt^d} (\eta_\delta*u) m^\epsi dx+\epsi \int_{\Tt^d} (\eta_\delta*u) dx-\epsi \int_{\Tt^d} u^\epsi dx +\int_{\Tt^d} (W-W^\epsi)m^\epsi dx.
\end{align}
Next, we integrate the right-hand side of \eqref{A0} and use Jensen's inequality to obtain
\begin{equation}
\label{A3}
\int_{\Tt^d\times \Rr^d} 
\frac{1}{2}|D(\eta_\delta*u)|^2+W d\mu^\epsilon
\leq 
\int_{\Tt^d\times \Rr^d} 
\frac{1}{2}\eta_\delta*(|Du|^2)+W d\mu^\epsilon
\leq 
\int_{\Tt^d\times \Rr^d}  -\eta_\delta*W+W d\mu^\epsilon, 
\end{equation}
taking into account \eqref{ssole}. 
Because 
$W\in L^{\frac{\alpha+1}{\alpha}}$, 
$\eta_\delta*W\to W$ in $ L^{\frac{\alpha+1}{\alpha}}$ as $\delta\to 0$. 
Therefore, taking into account that 
$m^\epsilon\in L^{\alpha+1}$, 
we have 
\[
\int_{\Tt^d\times \Rr^d}  -\eta_\delta*W+W d\mu^\epsilon \to 0, 
\] 
as $\delta\to 0$. 
Therefore, 
combining \eqref{A0}, \eqref{A1}, and \eqref{A3}, and 
by considering the limit
$\delta \to 0$, 
we conclude that 
\begin{equation}\label{PP1}
-\epsi \int_{\Tt^d} u m^\epsi dx+\epsi \int_{\Tt^d} u dx-\epsi \int_{\Tt^d} u^\epsi dx+\int_{\Tt^d} (W-W^\epsi)m^\epsi dx \leq 0.
\end{equation}
On the other hand, we observe that, for all $v\in \Rr^d$ and almost every $x \in \Tt^d$,
{
\begin{equation*}
\epsi u^\epsi+v\cdot D(\eta_\delta*u^\epsi)-L(x,v)+W^\epsi-W \leq \epsi u^\epsi+\frac{1}{2}|D(\eta_\delta*u^\epsi)|^2+W+W^\epsi-W.
\end{equation*}
}
Integrating with respect to $\mu$ and proceeding in a similar manner, we obtain
 \begin{equation}\label{PP2}
\epsi \int_{\Tt^d} u^\epsi m dx+\int_{\Tt^d} (W^\epsi-W)m dx \leq0.
\end{equation}
Next, from \eqref{PP1}, we gather
\begin{equation*}
-\epsi \int_{\Tt^d} \langle u \rangle m^\epsi dx-\epsi \int_{\Tt^d} u^\epsi dx+\int_{\Tt^d} (W-W^\epsi)m^\epsi dx \leq 0.
\end{equation*}
Finally, from \eqref{PP2}, we get
\begin{equation*}
\epsi \int_{\Tt^d} \langle u^\epsi \rangle m dx+\epsi \int_{\Tt^d} u^\epsi dx +\int_{\Tt^d} (W^\epsi-W)m dx \leq 0.
\end{equation*}


Adding the above two inequalities, we obtain
\begin{equation}
\label{RRR}
\epsi \left( \int_{\Tt^d} \langle u^\epsi \rangle m dx-\int_{\Tt^d} \langle u \rangle m^\epsi dx \right)+\int_{\Tt^d} (W-W^\epsi)(m^\epsi-m) dx \leq 0.
\end{equation}
By Poincar\'e's inequality, we have $\langle u^\epsi \rangle, \langle u \rangle  \in L^{2^*}$, uniformly in $\epsilon$. Moreover, $m, m^\epsilon\in L^{\frac{2^*}{2}(\alpha+1)}$, uniformly in $\epsilon$. Therefore,  
 $\int \langle u^\epsi \rangle m dx$ and $\int \langle u \rangle m^\epsi dx$ are bounded uniformly in  $\epsi$.
 Consequently,  the first term in the left-hand side of \eqref{RRR} converges to 0. Hence, we obtain \eqref{converge}.
Moreover, because the second term is non-negative, we conclude by the monotonicity of $g$ that
\[
\int_{\Tt^d} \langle u^\epsi \rangle m dx-\int_{\Tt^d} \langle u \rangle m^\epsi dx \leq 0. 
\]
 Hence, \eqref{criterion} holds.
\end{proof}

\subsection{An explicit example}
\label{SSEA}
Finally, in this subsection, we present an application of our selection criterion. We consider the following discount problem:
\begin{equation}\label{exDP}
        \begin{cases}
                &\epsi u^\epsi +\frac{1}{2}|u_x^\epsi|^2+\pi \cos(2\pi x)=m^\epsi \quad\rm{in}\  \Tt^1, \\
                & \epsi m^\epsi-(m^\epsi u_x^\epsi)_x = \epsi \quad \rm{in} \ \Tt^1, \\
        \end{cases}
\end{equation}
Thus $d=1$, $V(x)=\pi \cos(2\pi x)$, and $g(m)=m$.
The associated limit problems of \eqref{exDP} is the following:
\begin{equation}\label{exLP}
        \begin{cases}
                &\frac{1}{2}|u_x|^2+\pi \cos(2\pi x)=m+\Hh \quad\rm{in}\  \Tt^1, \\
                & -(mu_x)_x = 0 \quad \rm{in} \ \Tt^1, \\
                & m(x)\geq 0,\ \ \int_{\Tt^1} m \mbox{ }dx =1.\\
        \end{cases}
\end{equation}
 By Theorem \ref{ws}, there exists a regular weak solution $(u^\epsi,m^\epsi)\in H^1(\Tt^1)\times H^1(\Tt^1)$, of \eqref{exDP}.
We note that $u^\epsi -\int_{\Tt^d} u^\epsi dx$ converges along subsequence weakly in $H^1(\Tt^1)$. In view of Proposition \ref{stabilitypro}, the limit is a regular weak solution of \eqref{exLP}. However, regular weak solutions for \eqref{exLP} are not unique, as we show next.  

\begin{proposition}
Uniqueness of regular weak solutions for \eqref{exLP} does not hold. In particular, there exist regular weak solutions of \eqref{exLP}, $(\tilde u, m, \Hh)$ and $(\hat u,m,\Hh)$, with $\tilde u -\hat u$ not identically constant.
\end{proposition}

\begin{proof}
The second equation in \eqref{exLP} gives that $mu_x \equiv 0$ almost everywhere. Thus, 
\begin{equation*}
\label{intweaksol22}
\lambda-\pi \cos(2\pi x)+m=0, 
\end{equation*}
almost everywhere in $\{x\in \Tt : \: m(x)>0\}$.
Then,
\[ 1=\int_{\Tt^1} m \mbox{ }dx=\int_{\{m>0\}} m \mbox{ }dx=\int_{\Tt^1} (\pi \cos(2\pi x)-\lambda)_+ \mbox{ }dx.\]
Accordingly, $\lambda=0$ and $m(x)=(\pi \cos (2\pi x))_+$.

On the other hand, from the first equation in \eqref{exLP}, we see that $u\in H^1(\Tt^1)$is a regular weak solution if $u$ satisfies
\begin{equation}\label{deri}
u_x \equiv 0 \quad \mathrm{a.e.} \: \mathrm{on} \: \{0<x<1/4\}\cup\{3/4<x<1\},
\end{equation}
 and
\begin{equation}\label{derivative}
|u_x|^2 \leq -2\pi \cos(2\pi x) \quad \mathrm{a.e.} \: \mathrm{on}  \: \{1/4<x<3/4\}.
\end{equation}
For example, the functions $\tilde u$ and $\hat u$ determined by
\begin{align*}
 \tilde u_x(x)=&\sqrt{(-2\pi \cos(2\pi x))^+}\cdot \chi_{\{\frac{1}{4}<x<\frac{1}{2}\}}-\sqrt{(-2\pi \cos(2\pi x))^+}\cdot \chi_{\{\frac{1}{2}<x<\frac{3}{4}\}},
 \end{align*}
and
 \begin{align*}
 \hat u_x(x)=-\sqrt{(-2\pi \cos(2\pi x))^+}\cdot \chi_{\{\frac{1}{4}<x<\frac{3}{8}\}\cup\{\frac{1}{2}<x<\frac{5}{8}\}}+\sqrt{(-2\pi \cos(2\pi x))^+}\cdot \chi_{\{\frac{3}{8}<x<\frac{1}{2}\}\cup\{\frac{5}{8}<x<\frac{3}{4}\}},
 \end{align*}
satisfy \eqref{deri} and \eqref{derivative}.
Thus, $(\tilde u,m, \Hh)$ and $(\hat u,m, \Hh)$ are regular weak solutions.
\end{proof}
By \eqref{criterion}, we get the following result:

\begin{proposition}
Let  $(u^\epsi,m^\epsi)$ be regular weak solution of \eqref{exDP}. Suppose that $u^\epsi-\int u^\epsi dx \to \bar{u}$ as $\epsi \to 0$ weakly in $H^1(\Tt^1)$.
Let $(u,m)$ be any regular weak solution of \eqref{exLP}.
Then,
\begin{equation*}
\int_{\Tt^1} \langle \bar u\rangle m \mbox{ }dx \leq \int_{\Tt^1} \langle u\rangle m \mbox{ }dx.
\end{equation*}
\end{proposition}
Using the above criterion, we can show that $u^\epsi-\int_{\Tt^1} u^\epsi dx$ fully converges weakly in $H^1(\Tt^1)$ sense and we can detect the unique limit of $u^\epsi -\int_{\Tt^1} u^\epsi dx$, as we show now.

\begin{proposition} Let $\tilde u\in H^1(\Tt^1)$ be determined by
 \begin{align*}
 \tilde u_x=&\sqrt{(-2\pi \cos(2\pi x))^+}\cdot \chi_{\{1/4<x<1/2\}}-\sqrt{(-2\pi \cos(2\pi x))^+}\cdot \chi_{\{1/2<x<3/4\}},
 \end{align*}
 and $\tilde u(0)=0$. Then $\tilde u$ is the unique minimizer of $ \int_{\Tt^1} \langle u\rangle m \mbox{ } dx$ over all weak solutions $u$ of \eqref{exLP}.
 \end{proposition}
 
 \begin{proof}
 Let $u$ be any regular weak solution of \eqref{exLP}. Because the quantities $ \int_{\Tt^1} \langle u\rangle m \mbox{ }dx$ is invariant by addition of a constant to $u$, we can assume $u(0)=\tilde u(0)=0$, without loss of generality. Moreover, because of \eqref{deri} and by periodicity, we have $u(x)=\tilde u(x)=0$ for $x\in [0,1/4]\cup[3/4,1]$. Then,
 \begin{align*}
 \int_{\Tt^1} \langle u\rangle \mbox{ d} m
 &=\int_{\Tt^1} um\mbox{ }dx-\left(\int_{\Tt^1} u\mbox{ }dx \right) \left(\int_{\Tt^1} m\mbox{ }dx \right)\\
 &=\int_{[1/4,3/4]} um\mbox{ }dx-\int_{\Tt^1} u\mbox{ }dx=-\int_{\Tt^1} u\mbox{ }dx.
 \end{align*}
 Hence, it suffices to discuss the quantities of $-\int_{\Tt^1} u\mbox{ }dx$.
 
 Because of \eqref{derivative}, we can see that $\tilde u(x) \geq u(x)$ in $x\in [1/4,1/2]$. On the other hand, it holds that  $\tilde u(x) \geq u(x)$ in $x\in [1/2,3/4]$. Indeed, suppose that there exists $x_0 \in [1/2,3/4]$ and solution $u_0$ such that $\tilde u(x_0) < u_0(x_0)$. Then, it follows from \eqref{derivative} that $u_0(3/4)>\tilde u(3/4)=0$, which is a contradiction.
 Thus, $\tilde u(x)\geq u(x)$ for $x\in \Tt^1$ and we see $\int_{\Tt^1}\tilde u\mbox{ }dx\geq \int_{\Tt^1}u\mbox{ }dx$.
 \end{proof}


\bibliographystyle{plain}

\IfFileExists{"/Users/gomesd/mfgDGOFFICE.bib"}{\bibliography{/Users/gomesd/mfgDGOFFICE.bib}}{\bibliography{mfg.bib}}

\end{document}